%% file: main-extremes.tex
\documentclass[pdflatex,sn-mathphys-ay]{sn-jnl}

\usepackage{graphicx}%
\usepackage{multirow}%
\usepackage{amsmath,amssymb,amsfonts}%
\usepackage{amsthm}%
\usepackage{mathrsfs}%
\usepackage[title]{appendix}%
\usepackage{xcolor}%
\usepackage{textcomp}%
\usepackage{manyfoot}%
\usepackage{booktabs}%
\usepackage{algorithm}%
\usepackage{algorithmicx}%
\usepackage{algpseudocode}%
\usepackage{listings}%

\usepackage{hyperref}
\let\orcidlogo\relax
\usepackage{orcidlink}
\hypersetup{hidelinks}
\usepackage[capitalise]{cleveref}
\usepackage[bb=dsserif]{mathalpha}

\theoremstyle{thmstyleone}%
\newtheorem{theorem}{Theorem}
\theoremstyle{thmstyletwo}%
\newtheorem{remark}{Remark}%

\theoremstyle{thmstylethree}%
\newtheorem{definition}{Definition}%

\newcommand{\R}{\mathbb{R}}
\newcommand{\bigO}{\mathcal{O}}

\newtheorem{lemma}{Lemma}
\newtheorem{corollary}{Corollary}
\newcommand{\1}{ \mathbb{1} }
\renewcommand{\P}{ \mathbb{P} }

\newcommand{\E}{ \mathbb{E} }

\begin{document}

\title[Large and Moderate Deviations for Conservative Tail-Index Estimation]{Large and Moderate Deviations for Conservative Tail-Index Estimation}


\author*[1]{\fnm{Martijn} \sur{Gösgens}\,\orcid{https://orcid.org/0000-0002-7197-7682}}\email{research@martijngosgens.nl}

\author[2]{\fnm{Bart P.G.} \sur{van Parys}\,\orcid{https://orcid.org/0000-0003-4177-4849}}\email{bart.van.parys@cwi.nl}

\author[2,3]{\fnm{Bert} \sur{Zwart}\,\orcid{https://orcid.org/0000-0001-9336-0096}}\email{bert.zwart@cwi.nl}

\affil[1]{University of Twente, Enschede, The Netherlands}
\affil[2]{CWI, Amsterdam, The Netherlands}
\affil[3]{TU Eindhoven, Eindhoven, The Netherlands}

\abstract{
\input{abstract.tex}
}

\keywords{Large deviations, Extreme value theory, Regular variation, Hill estimator}


\pacs[MSC Classification]{
62G32
60F10
62G20
62G05
62G30
60G70
}

\maketitle

\input{body.tex}

\backmatter

\section*{Statements and declarations}

\bmhead{Funding}

Bart P.G.\ van Parys gratefully acknowledges funding from NWO Vidi grant VI.Vidi.243.021.
The work of Martijn Gösgens was supported by the Netherlands Organisation for Scientific Research (NWO) through the Gravitation NETWORKS grant no. 024.002.003.

\bmhead{AI statement}

The authors used OpenAI's ChatGPT and Anthropic's Claude for exploratory calculations and copy-editing. In particular, ChatGPT assisted exploratory work leading to the proof of \cref{thm:hill-bias-corrected}. The authors independently verified and rewrote all arguments and take full responsibility for the final manuscript.

\bmhead{Competing interests}

The authors declare no competing interests.

\begin{appendices}
\renewcommand{\theHequation}{\thesection.\arabic{equation}}

\input{appendix.tex}
\end{appendices}


\makeatletter\immediate\write\@auxout{\string\citation{snsettings}}\makeatother
\bibliography{main}

\end{document}

%% file: abstract.tex
    To design systems that are protected against events much rarer than the observational record, extreme-value methods are needed to extrapolate distribution tails. Tail-index estimators such as the Hill estimator are central to this extrapolation, but overestimating the tail exponent can lead to substantial underestimation of rare-event probabilities. Motivated by this, we derive large- and moderate-deviation asymptotics for the Hill estimator and use them to construct estimators whose probability of exceeding the true tail index decays at a controlled exponential rate (the \emph{decay rate}). 
    In the large-deviations regime, we show that a simple rescaled version of the Hill estimator achieves an optimal balance between bias and decay rate among scale-invariant estimators based on the same top $k$ order statistics.
    Under a second-order condition, we quantify the effect of the Hill bias, analyze a bias-corrected estimator, and identify sufficient conditions for moderate deviations in the boundary case where the second-order parameter $\rho$ equals zero.

%% file: body.tex
\section{Introduction}
Dams, dikes, and insurance reserves are designed for floods and losses of sizes that may not have been observed before. Extreme value theory makes this possible by extrapolating the tail of a distribution beyond the range of the data~\citep{haan2006extreme,beirlant2006statistics,Embrechts2013ExtremeValueTheoryFinanceInsurance}. For distributions with regularly varying tails, the tail probability $\P(X>x)$ decays as $x^{-\alpha}$ up to a slowly varying factor, and the $(1-\delta)$-quantile grows as $\delta^{-1/\alpha}$. The \emph{tail index} $\alpha>0$ thus determines the extrapolation to leading order and needs to be estimated from data. The most common estimator for $\alpha$ is the Hill estimator~\citep{hill1975simple}, which estimates $\alpha$ using the $k$ largest observations of $n$ i.i.d.\ data points.


The statistical properties of the Hill estimator have been studied extensively, including its consistency~\citep{mason1982laws} and asymptotic normality under suitable regularity conditions~\citep{haan1998asymptotic}. In this paper, we investigate large and moderate deviations properties, which are not as well understood. Our motivation to study such properties is that errors in the tail index can have a large effect on extrapolated probabilities. If an estimator for $\alpha$ returns $(1+\theta)\alpha$, then the extrapolated probability of an event of size $x$ changes from $x^{-\alpha}$ to $x^{-(1+\theta)\alpha}$, so that the estimate is off by a factor $x^{\theta\alpha}$, which grows unboundedly; see~\cite{Smith_Weissman_1987} for an early warning on such extrapolation errors. The two directions of error are therefore not equally harmful. If we overestimate $\alpha$, then we underestimate the probability of catastrophically large observations by orders of magnitude, and a system designed on the basis of this estimate is not as resilient as intended. If we underestimate $\alpha$, then we overestimate the probability of catastrophes, and the system is conservative but safe. This motivates estimators for which the probability of overestimating $\alpha$ decays exponentially, at a rate that we call the \emph{decay rate} of the estimator (see~\cref{def:conservative}). To find out which decay rates are achievable, and at what bias, we need the large and moderate deviations of tail index estimators.

We derive the large and moderate deviations of the Hill estimator for regularly varying distributions and use them to construct estimators with a given decay rate. We show that a rescaled Hill estimator achieves a decay rate that no scale-invariant estimator based on the same order statistics can improve without increasing its bias, so in this context, the Hill estimator is optimal. In the moderate deviations regime, we use a classical second-order condition to identify which deviation sizes $\theta_k$ are admissible. 

\subsection{Background and preliminaries}
We write $g\in\mathcal{RV}_\gamma$ and say that the function $g$ is \emph{regularly varying} with exponent $\gamma$ whenever
\[
\frac{g(tx)}{g(t)}\longrightarrow x^\gamma,
\]
for every $x>0$ as $t\to\infty$.
Let $X,X_1,\ldots,X_n$ be independent copies of a positive random variable with distribution function $F$. We denote the order statistics of $X_1,\dots,X_n$ (in decreasing order) by
\[
X_{1,n}\ge X_{2,n}\ge\cdots\ge X_{n,n}
\]
and abbreviate $X_{1:k,n}:=(X_{1,n},\ldots,X_{k,n})$. For $t>1$, the tail quantile function is
\[
U(t):=\inf\{u\ : F(u)\ge 1-1/t\}.
\]
The tail quantile function $U$ is nondecreasing and left-continuous.
Let $\mathcal{RV}'_\gamma$ denote the set of positive, non-decreasing, left-continuous regularly varying functions with exponent $\gamma$.
Throughout, $\text{Exp}(\lambda)$ denotes the exponential distribution with rate parameter $\lambda$.
We say that the distribution of $X$ is regularly varying with exponent $\alpha>0$ if $1-F\in\mathcal{RV}_{-\alpha}$, or equivalently~\citep{haan2006extreme,nair_fundamentals_2022}, $U\in\mathcal{RV}'_{1/\alpha}$. We write $X\sim \text{RV}(U)$ to denote that  $X\stackrel d=U(e^E)$ for $E\sim\text{Exp}(1)$.
The class of regularly varying distributions forms the Maximum Domain of Attraction of the Fréchet distribution. That is~\citep{Haan1970regular}, $U\in\mathcal{RV}'_{1/\alpha}$ if and only if for all $x>0$,
\[
\lim_{n\to\infty}\P(X_{1,n}/U(n)\le x)= e^{-x^{-\alpha}}.
\]

All asymptotic statements are understood along an arbitrary intermediate sequence $k=k_n$ satisfying
\[
k_n\longrightarrow\infty,\qquad \frac{k_n}{n}\longrightarrow0.
\]
We retain the shorthand $n\gg k\gg1$ for this regime. Sequences denoted by $\theta_k$, $A(n/k)$, and similar expressions are evaluated along the same sequence $k=k_n$.

Whenever a reciprocal estimator $c/T$ is used below, we adopt the convention $c/T:=+\infty$ on $\{T\le0\}$. With this convention, reciprocal inequalities remain valid without a separate positivity event.

\begin{definition}\label{def:rescaled-hill}
    For $\theta\in[0,1)$, define the \emph{rescaled Hill estimator}
    \[
    \hat\alpha^{(\theta)}_{k,n}:=\frac{1-\theta}{H_{k,n}},
    \qquad
    H_{k,n}:=\frac1k\sum_{i=1}^{k}\log\left(\frac{X_{i,n}}{X_{k,n}}\right).
    \]
\end{definition}
For $\theta=0$, this is the classical Hill estimator~\citep{hill1975simple}, which is the most widely used tail index estimator.
The Hill estimator is consistent~\citep{haan2006extreme}, i.e. $\hat\alpha^{(0)}_{k,n}\stackrel\P\to\alpha$.

More precise asymptotic statements are available under the second-order condition~\citep{haan1998asymptotic}:
\begin{definition}\label{def:soc}
    Consider an auxiliary function $A\in\mathcal{RV}_\rho$ for some \emph{second-order parameter} $\rho\le0$, with $A(t)\to0$ and $A(t)\neq0$ for sufficiently large $t$.
    We write $U\in2\mathcal{RV}_{1/\alpha,A}$ and say that $U$ satisfies the \emph{second-order condition} with rate $A(t)$ if
    \begin{equation*}\label{eq:soc}\tag{SOC}
    \lim_{t\to\infty}\frac{U(tx)/U(t)-x^{1/\alpha}}{A(t)}
    =x^{1/\alpha}\int_1^x u^{\rho-1}\,du,
    \end{equation*}
    for all $x>0$.
\end{definition}
Theorem~3.2.5 of~\citet{haan2006extreme} shows that if~\eqref{eq:soc} holds and additionally $\sqrt kA(n/k)=\bigO(1)$, then
\begin{equation}\label{eq:hill-normality}
\sqrt k\left(\alpha H_{k,n}-1-\frac{\alpha A(n/k)}{1-\rho}\right)\stackrel d\to N(0,1).
\end{equation}
Notice that if $A(n/k)<0$ and $\sqrt k|A(n/k)|$ is large, then $\P(\hat\alpha^{(0)}_{k,n}=1/H_{k,n}>\alpha)\approx1$, so that the Hill estimator is likely to overestimate the tail index.
As discussed above, overestimating the tail index leads to underestimating the probability of extreme events by orders of magnitude. This motivates the following definition:
\begin{definition}\label{def:conservative}
    Let $\lambda_n\to\infty$. We call an estimator $\hat\alpha_n:\mathbb{R}^n\to[0,\infty]$ \emph{conservative with decay rate at least $\lambda_n$} if, for any $\alpha>0, U\in\mathcal{RV}'_{1/\alpha}$ and $X_1,\dots,X_n\sim\text{RV}(U)$,
    \begin{equation}\label{eq:conservative-single}
    \liminf_{n\to\infty}\frac{-1}{\lambda_n}\log\P(\hat\alpha_n(X_1,\dots,X_n)>\alpha)\ge1.
    \end{equation}
    If the limit exists and equals $1$ for every regularly varying distribution, we say that the estimator \emph{achieves} the decay rate $\lambda_n$.
\end{definition}
For the estimators that we study, the asymptotics of $-\log\P(\hat\alpha>\alpha)$ actually does not depend on the particular $\alpha$ or regularly varying distribution, which is why we define the decay rate in terms of this exponential rate rather than directly studying the asymptotics of $\P(\hat\alpha>\alpha)$, which is much more sensitive to the particular regularly varying distribution at hand.

The convergence in~\cref{def:conservative} is \emph{pointwise} in $U\in\mathcal{RV}'_{1/\alpha}$. If we were to require \emph{uniform} convergence over $U\in\mathcal{RV}'_{1/\alpha}$, then this could only be achieved by pathological tail estimators, as we show below. Consider an arbitrary quantile function $U'$ and a threshold $0\le t_0<\infty$. We construct the quantile function
\[
U_{t_0}(t)=\begin{cases}
U'(t)&\text{ if }t<t_0,\\
U'(t_0)\cdot(t/t_0)^{1/\alpha}&\text{ if }t\ge t_0.
\end{cases}
\]
Notice that $U_{t_0}(t)\in\mathcal{RV}'_{1/\alpha}$ regardless of $U',t_0$. Recall that $X_i\stackrel d=U(e^{E_i})$ for independent $E_i\sim\text{Exp}(1).$ The probability that there is some $i\in[n]$ with $e^{E_i}>t_0$ is at most $n/t_0$, so that $X_i=U'(e^{E_i})$ holds for all $i\in[n]$ with probability at least $1-n/t_0$. Hence, for fixed $n$ and $t_0\to\infty$, 
\begin{equation}\label{eq:conservative-uniform}
\lim_{t_0\to\infty}\P(\hat\alpha_n(X_1,\dots,X_n)>\alpha)=\P(\hat\alpha_n(U'(e^{E_1}),\dots,U'(e^{E_n}))>\alpha).
\end{equation}
If $\hat\alpha_n$ would satisfy a uniform variant of~\cref{def:conservative}, then the right-hand-side of~\eqref{eq:conservative-uniform} would have to vanish as $n\to\infty$. By picking $\alpha>0$ arbitrarily small, we conclude $\hat\alpha_n(X_1',\dots,X_n')\stackrel\P\to0$ for $X_1',\dots,X_n'\sim\text{RV}(U')$ and arbitrary $U'$. Hence, a uniform variant of~\cref{def:conservative} could only hold for estimators $\hat\alpha_n$ with $\hat\alpha_n\stackrel\P\to0$ regardless of the distribution.

\subsection{Main results}\label{sec:main-results}
Our first main result concerns the large and moderate deviations of the Hill estimator: for $\theta_k<1$ and with some additional conditions (given in \cref{thm:main-hill}), we prove
\begin{equation}\label{eq:hill-deviations}
-\log\P(\alpha H_{k,n}<1-\theta_k)\sim kf(\theta_k),\quad -\log\P(\alpha H_{k,n}>1+\theta_k)\sim kf(-\theta_k),
\end{equation}
for the rate function
\begin{equation}\label{eq:rate-function}
f(\theta):=-\theta-\log(1-\theta).
\end{equation}
Notice that $f(\theta)\ge0$, with equality only at $\theta=0$, and $f(\pm \theta)\sim \theta^2/2$ as $\theta\downarrow0$. The reason that we define $f(\theta)$ to match the lower tail in~\eqref{eq:hill-deviations} is that the lower tail is particularly relevant for conservative estimation of the tail index, as will become clear later.

For $\theta_k=\theta$ constant, \eqref{eq:hill-deviations} gives the large deviations of the Hill estimator (see~\cref{cor:hill-ld}). This extends a result by \citet{shihong1992large}, who derived the large deviations of the Hill estimator under the assumption that the distribution function $F$ is continuous. Our proof is simpler and shows that this continuity assumption is not necessary.

For the moderate deviations regime, we rely on the second-order condition~\eqref{eq:soc}. This is a natural assumption, since moderate deviations lie between the large deviations and CLT behavior, and the CLT behavior requires the second-order condition~\citep{haan1998asymptotic}.
\eqref{eq:soc} comes with a second-order parameter $\rho\le0$ and an auxiliary function $A\in\mathcal{RV}_\rho$.
We distinguish the case $\rho<0$ from the case $\rho=0$, since the latter case is more involved.
For $\rho<0$, we show that~\eqref{eq:hill-deviations} holds for $\theta_k\gg |A(n/k)|+k^{-1/2}$.
For $\rho=0$, \cref{thm:hill-rho0} proves~\eqref{eq:hill-deviations} for $\theta_k\gg \sup_{t\ge n/k}|A(t)|^{1/(1+\kappa)}+k^{-1/2}$, for $\kappa>0$.

To also obtain moderate deviations for $\theta_k\asymp |A(n/k)|$ and $\rho<0$, we need to take the sign of $A(n/k)$ into account. \cref{thm:hill-bias-corrected} proves the moderate deviations regime of a bias-corrected Hill estimator $\tilde H_{k,n}=H_{k,n}-\frac{A(n/k)}{1-\rho}$, which can be translated to moderate deviations of $H_{k,n}$.
Our moderate deviations results are valid even when $\sqrt{k}|A(n/k)|\to\infty$, in contrast to the CLT behavior of the Hill estimator, which requires $\sqrt kA(n/k)=\bigO(1)$.

Our main motivation for studying the large and moderate deviations of the Hill estimator, is to construct tail-index estimators with a controlled decay rate. Since $\{\alpha H_{k,n}<1-\theta_k\}=\{\hat\alpha_{k,n}^{(\theta_k)}>\alpha\}$, condition \eqref{eq:hill-deviations} implies that the rescaled Hill estimator $\hat\alpha_{k,n}^{(\theta_k)}$ achieves the decay rate (\ref{eq:conservative-single}) with $\lambda_k=kf(\theta_k)$.

In the large deviations regime, we have $kf(\theta)=\Theta(k)$.
By the consistency of the Hill estimator, we have
\[
\hat\alpha_{k,n}^{(\theta)}\stackrel\P\to(1-\theta)\alpha.
\]
This shows that the $\Theta(k)$ decay rate of $\hat\alpha_{k,n}^{(\theta)}$ comes at the cost of a non-vanishing bias. This raises the question whether it is possible to have an estimator that is strictly less biased than the rescaled Hill estimator while achieving the same decay rate.
\cref{thm:hill-optimality} shows that among \emph{scale-invariant estimators} that make use of the top $k$ order statistics, this is not possible. That is, for any estimator $\hat\alpha_{k,n}:[0,\infty)^k\to[0,\infty]$ for which $\hat\alpha_{k,n}(cX_{1:k,n})=\hat\alpha_{k,n}(X_{1:k,n})$ for all $c>0$, if $\liminf\P(\hat\alpha_{k,n}>(1-\theta)\alpha)>0$ for some normalized regularly varying distribution (see~\cref{def:nrv}), then its decay rate is at most $kf(\theta)+o(k)$.
However, \cref{lem:non-scale-invariant} shows that outside the class of scale-invariant estimators, the estimator $(1-\varepsilon)\log(n/k)/\log X_{k,n}$ with $\varepsilon<\theta$ has a smaller bias than $\hat\alpha^{(\theta)}$, but decay rate $\lambda_k\gg k$. Most of the estimators that are used in practice are scale-invariant, since otherwise the estimate would depend on the unit of measurement, i.e., whether losses are expressed in dollars or millions of dollars.
In \cref{thm:pickands}, we show that a rescaled version of the (scale-invariant) Pickands estimator~\citep{pickands1975statistical} has a strictly lower decay rate at the same bias level, when making use of the same number of order statistics.
\cref{thm:hill-optimality} also highlights that a consistent scale-invariant estimator for $\alpha$ must have decay rate $o(k)$, which indeed matches the moderate deviations regime $\theta_k\downarrow0$.

In this work, we consider a given sequence $k_n$ and leave the separate question of how $k_n$ should be chosen for future work. From \eqref{eq:hill-deviations}, it is already clear that if we want to achieve a given decay rate $\lambda_n$, then we need $k_n=\Omega(\lambda_n)$. To obtain a consistent estimator, we even need $k_n\gg \lambda_n$. If additionally, we know what bias level $\theta_n$ is tolerable, then a natural choice would be $k_n\approx \lambda_n/f(\theta_n)$. In particular, for $\theta_n\downarrow0$, we have $k_n\sim 2\lambda_n/\theta_n^2$.

\subsection{Related work}
\paragraph{The Hill estimator and large deviations.}
The Hill estimator was introduced in~\cite{hill1975simple} and its consistency for $1\ll k\ll n$ was proven in~\citet{mason1982laws}. \citet{deheuvels1988strong} proves almost sure convergence for $\log\log n\ll k\ll n$.
The asymptotic normality was established in~\cite{haan1998asymptotic} under the \eqref{eq:soc} assumption.
\citet{haan1998comparison} compare the Hill estimator against several other estimators, including Pickands'~\citep{pickands1975statistical}, in terms of asymptotic mean squared error, and find that the MSE-optimal $k$ differs across estimators.
\citet{Lu_Peng_2002} build likelihood-based confidence intervals for $\alpha$ around the Hill estimator. Their results concern fixed confidence levels, i.e., the typical fluctuations of the estimator, rather than the exponentially small error probabilities that we are interested in.

\citet{shihong1992large} proves the exact large-deviation asymptotics of the Hill estimator under the assumption that $F$ is continuous. \cref{cor:hill-ld} extends this result to the case where $F$ is not necessarily continuous.
\citet{viharos2008large} derives large-deviation probabilities for a linear class of tail index estimators, and shows that the Hill estimator achieves the optimal large-deviations rate of convergence for Pareto-distributed data. Our optimality result from \cref{thm:hill-optimality} is more general since it applies to arbitrary \emph{scale-invariant} functions of the $k$ largest order statistics and applies to regularly varying distributions instead of just Pareto.

\paragraph{Bias correction.}
In a separate line of work, the bias of the Hill estimator is reduced by estimating the second-order parameter $\rho$ and the second-order rate $A(n/k)$.
\citet{Ivette_Gomes_Pestana_2007} present methods to perform such a bias-correction in the Hall-class~\citep{hall1982simple}  (i.e., \eqref{eq:soc} with $\rho<0$ and $A(t)\sim (\beta/\alpha)t^\rho$ for some constant $\beta$) by introducing consistent estimators for $\rho$ and $\beta$. They prove asymptotic normality of their bias-corrected Hill estimator and discuss how to choose $k_n$ to minimize the asymptotic mean squared error.
\citet{Caeiro_Gomes_Rodrigues_2009} extend this estimator to a third-order condition and derive the asymptotic distribution of this estimator.
In \cref{thm:hill-bias-corrected}, we derive the moderate deviations of the Hill estimator minus the true value of the bias term, without providing methods for estimating this bias term from the data. Deriving the moderate deviations of the estimators from \citet{Ivette_Gomes_Pestana_2007} or \citet{Caeiro_Gomes_Rodrigues_2009} is significantly more challenging.

\paragraph{Extreme quantile extrapolation.} Once a tail index has been estimated, the standard route to an extrapolated extreme quantile combines an intermediate empirical quantile $X_{k,n}$ with the tail-index estimate $\hat\alpha_{k,n}$, an idea originating from \citet{Weissman_1978}.
\citet{Allouche_ElMethni_Girard_2023} show that in the second order framework, it is possible to reduce the bias of such a quantile estimator by picking different values $k$ for these two estimators.
\citet{Danielsson_Ergun_Haan_Vries_2025} instead select the number of order statistics by minimizing the discrepancy between the fitted Pareto-type tail and the empirical quantiles. These works optimize the accuracy of the resulting extrapolation, while our results instead control the probability that the tail estimate is too optimistic via the decay rate.

\paragraph{Robust estimation.}
A separate line of work protects decisions against model uncertainty by optimizing over an ambiguity set of plausible distributions rather than basing decisions on a single distribution~\citep{Ben-Tal_Nemirovski_2009,Bertsimas_Gupta_Kallus_2018}. 
Within extreme value theory, \citet{Blanchet_He_Murthy_2020,Birghila_Aigner_Engelke_2025} study distributionally robust estimation of tail probabilities. That is, they consider the worst-case tail probability among an ambiguity set $\sup\{\P(X>x)\ :\ d(\P,\hat\P)\le\delta\}$, where $d$ is a divergence measure and $\hat\P$ is a reference distribution.
 \citet{Blanchet_He_Murthy_2020} shows that for Rényi divergence and $\hat\P$ regularly varying with exponent $\alpha_{\mathrm{ref}}$, the worst-case tail is again regularly varying, but with index $\alpha^*=\frac{\eta-1}\eta\alpha_{\mathrm{ref}}$ for Rényi parameter $\eta>1$, while the Kullback--Leibler case ($\eta=1$) leads to tails that are heavier than any regularly varying distribution. 
Most similar to the task of estimating the tail index with a controlled decay rate is the work of \citet{Parys_Zwart_2025}, who study robust estimation of the mean of a heavy-tailed random variable. They prove that for a Kullback--Leibler ambiguity-set estimator, the probability of overestimating the true mean decays at an exponential rate (analogous to our decay rate) and minimizes bias among all estimators achieving the same guarantee. 

\subsection{Organization of this paper}
\cref{sec:coupling-concentration} proves~\cref{thm:main-hill}, our main technical
result, using a coupling and concentration argument. 
\cref{sec:ld} specializes~\cref{thm:main-hill} to the large deviations regime. It derives the decay rate of the rescaled Hill estimator and shows that no scale-invariant estimator using the same
order statistics improves on it, while an estimator outside this class does.
\cref{apx:pickands} analyzes the rescaled Pickands estimator and shows that its
decay rate is strictly smaller.
\cref{sec:md} derives the moderate deviations of the Hill estimator and a bias-corrected variant under second-order regular variation, and discusses the trade-off between consistency and decay rate of the corresponding rescaled Hill estimators.

\section{Coupling and concentration}\label{sec:coupling-concentration}
In this section, we prove~\cref{thm:main-hill}, which gives the large and moderate deviations of the Hill estimator. 
We rewrite the Potter bounds in the following way:
\begin{lemma}\label{lem:potter-rewritten}
    Let $U\in\mathcal{RV}_{1/\alpha}$ and define the \emph{Potter rate} $\varepsilon(t)$ as
    \[
    \varepsilon(t)=\inf\left\{\varepsilon>0\ :\ \forall x\ge y\ge t:\ (1-\varepsilon)\left(\frac{x}{y}\right)^{\frac{1-\varepsilon}\alpha}\le\frac{U(x)}{U(y)}\le(1+\varepsilon)\left(\frac{x}{y}\right)^{\frac{1+\varepsilon}\alpha}\right\}.
    \]
    Then $\lim_{t\to\infty}\varepsilon(t)=0$.
\end{lemma}
\begin{proof}
    Pick any $\varepsilon>0$. By the Potter bounds, there is some $t_0<\infty$ so that for all $x,y\ge t_0$, it holds that
    \[
    (1-\varepsilon)\left(\frac{x}{y}\right)^{\frac{1-\varepsilon}\alpha}\le\frac{U(x)}{U(y)}\le(1+\varepsilon)\left(\frac{x}{y}\right)^{\frac{1+\varepsilon}\alpha},
    \]
    which implies $\varepsilon(t)\le\varepsilon$ for all $t\ge t_0$. Since $\varepsilon$ was arbitrary, it follows that $\varepsilon(t)\to0.$
\end{proof}

Notice that from the definition of $\varepsilon(t)$, it follows that $\varepsilon(t)$ is non-increasing. That is, $t_1\le t_2$ implies $\varepsilon(t_1)\ge\varepsilon(t_2)$.
The main result of this section is the following theorem, which shows that deviations of size $\theta_k$ satisfying $\theta_k\gg\varepsilon(n/(\eta k))$ for some $\eta>1$ occur at rate $kf(\theta_k)$. In~\cref{sec:ld,sec:md}, we will describe the large and moderate deviations that are implied by this theorem.
\begin{theorem}\label{thm:main-hill}
    Consider a sequence $\theta_k=\bigO(1)$ such that $\theta_k\gg\varepsilon(n/(\eta k))+k^{-1/2}$ for some $\eta>1$. Then
    \[
    -\log\P(\alpha H_{k,n}>1+\theta_k)\sim kf(-\theta_k),
    \]
    as $n\gg k\gg 1$, where the rate function is given by 
    \[f(\theta)=-\theta-\log(1-\theta)\ge0.\]
    Moreover, if additionally $\limsup_{k\to\infty}\theta_k<1$, then
    \[
    -\log\P(\alpha H_{k,n}<1-\theta_k)\sim kf(\theta_k).
    \]
\end{theorem}
The reason for the variable $\eta>1$ is technical, but will not be relevant for the main large and moderate deviations results in \cref{sec:ld,sec:md}.
We write $f(-\theta_k)$ for the upper tail and $f(\theta_k)$ for lower tail because the lower tail is particularly relevant for conservative estimation of the tail index, as we explain in~\cref{sec:ld}.
Notice that for $\theta_k\downarrow0$, we have $f(\pm\theta_k)\sim \tfrac12\theta_k^2$.

To prove \cref{thm:main-hill}, we will first couple the order statistics $X_{1:k,n}$ to i.i.d.\ standard exponentials $\Delta_{1:k}$ and introduce several lemmas.
If $X\sim\text{RV}(U)$, then $X\stackrel d=U(e^E)$ for $U\in\mathcal{RV}'_{1/\alpha}$ and $E\sim\text{Exp}(1)$.
Notice that $e^E\sim\text{Pareto}(1)$. However, since our proof technique relies on exponential spacings, it is more convenient to work with exponential rather than Pareto random variables.
If $X_{k,n}$ denotes the $k$-th largest among $n$ independent samples from the same distribution, then $X_{k,n}\stackrel d=U(e^{E_{k,n}})$, where $E_{k,n}$ is the $k$-th largest observation among $n$ independent standard exponentials. Since $E_{k,n}-E_{k+1,n}\sim\text{Exp}(k)$ and $E_{k,n}-E_{k+1,n}\perp E_{k-1,n}-E_{k,n}$, we can couple $E_{1:n,n}$ to $n$ i.i.d. standard exponentials $\Delta_{1},\dots,\Delta_n$ using the Rényi representation~\citep{renyi1953theory} of order statistics:
\[
E_{k,n}=\sum_{i=k}^n\frac{\Delta_i}{i}.
\]
Combining these two, we can couple the order statistics $X_{1:n,n}$ to $\Delta_{1:n}$ by
\begin{equation}\label{eq:exp-spacing-coupling}
X_{k,n}=U\left(\exp\left(\sum_{i=k}^n\frac{\Delta_i}{i}\right)\right).
\end{equation}
Note that $E_{k,n}=\sum_{i=k}^n\Delta_i/i$ is a function of $\Delta_k,\dots,\Delta_n$ only, and is therefore independent of $\Delta_1,\dots,\Delta_{k-1}$. In particular, the spacings $E_{i,n}-E_{k,n}=\sum_{j=i}^{k-1}\Delta_j/j$, $i\in[k-1]$, are independent of $E_{k,n}$. We will use this repeatedly to pass from probabilities conditional on events involving $E_{k,n}$ to unconditional probabilities of events involving $\Delta_{1:k-1}$.
We start by bounding the Hill estimator in terms of the sum of i.i.d.\ $\text{Exp}(1)$ random variables $\Delta_i$ and the Potter rate:

\begin{lemma}\label{lem:hill-bound}
    Under the event $\{E_{k,n}\ge b\}$, we have
    \[
    \alpha\log(1-\varepsilon)+\left(1-\varepsilon\right)\frac1k\sum_{i=1}^{k-1}\Delta_i\le \alpha H_{k,n}\le \alpha\log(1+\varepsilon)+\left(1+\varepsilon\right)\frac1k\sum_{i=1}^{k-1}\Delta_i,
    \]
    where $\varepsilon=\varepsilon(e^{b})$.
\end{lemma}
\begin{proof}
    Write $y_i=E_{i,n}-E_{k,n}\ge0$ for $i\in[k]$ (so $y_k=0$). On $\{E_{k,n}\ge b\}$, for every $i\in[k]$ we have $x:=e^{E_{i,n}}\ge y:=e^{E_{k,n}}\ge e^{b}$, so \cref{lem:potter-rewritten} with $\varepsilon=\varepsilon(e^{b})$ applies to the pair $(x,y)$:
    \[
    (1-\varepsilon)\left(\frac xy\right)^{\frac{1-\varepsilon}\alpha}\le\frac{U(x)}{U(y)}\le(1+\varepsilon)\left(\frac xy\right)^{\frac{1+\varepsilon}\alpha}.
    \]
    Since $U(x)/U(y)=X_{i,n}/X_{k,n}$ and $x/y=e^{y_i}$, taking logs and multiplying by $\alpha$ gives
    \[
    \alpha\log(1-\varepsilon)+(1-\varepsilon)y_i\ \le\ \alpha\log\frac{X_{i,n}}{X_{k,n}}\ \le\ \alpha\log(1+\varepsilon)+(1+\varepsilon)y_i,
    \]
    for every $i\in[k]$. Summing over $i=1,\dots,k$ and dividing by $k$ yields
    \[
    \alpha\log(1-\varepsilon)+\frac{1-\varepsilon}k\sum_{i=1}^{k-1}y_i\ \le\ \alpha H_{k,n}\ \le\ \alpha\log(1+\varepsilon)+\frac{1+\varepsilon}k\sum_{i=1}^{k-1}y_i.
    \] 
    Since $y_i=\sum_{j=i}^{k-1}\Delta_j/j$, we have the identity
    \[
    \sum_{i=1}^{k-1}y_i=\sum_{i=1}^{k-1}\sum_{j=i}^{k-1}\frac{\Delta_j}j=\sum_{j=1}^{k-1}\frac{\Delta_j}j\sum_{i=1}^{j}1=\sum_{j=1}^{k-1}\Delta_j.
    \]
    This completes the proof.
\end{proof}

We will condition on the event $\{E_{k,n}\ge\log(n/k)-\log\eta\}$ for suitable $\eta>1$, so that we can apply the bounds from \cref{lem:potter-rewritten} with $\varepsilon(n/(\eta k))$.

\begin{lemma}\label{lem:intermediate-ld}
Let $\lambda_k=\bigO(k)$ and $\eta>1$ be such that $\limsup_{k\to\infty}\lambda_k/k<\eta-1-\log\eta$. Then
    \[
    \liminf_{n\gg k\gg 1}\frac{-1}{\lambda_k}\log\P\left(E_{k,n}<\log(n/k)-\log\eta\right)>1.
    \]
    In particular, for any $\lambda_k=\bigO(k)$ there exists some $\eta>1$ for which the above holds, and for $\lambda_k=o(k)$ it holds for \emph{any} $\eta>1$.
\end{lemma}
\begin{proof}
    Notice that the event $\{E_{k,n}<\log(n/k)-\log\eta\}$ means that fewer than $k$ out of $n$ independent exponentials have value at least $\log(n/k)-\log\eta$. This allows us to rewrite this probability to a binomial tail and apply the Chernoff--Hoeffding inequality using $\P(E\ge\log(n/k)-\log\eta)=\eta k/n> k/n$:
    \[
    \P\left(E_{k,n}<\log(n/k)-\log\eta\right)=\P(\text{Bin}(n,\eta k/n)<k)\le e^{-nD(k/n\|\eta k/n)},
    \]
    where $D(x\|y)=x\log\tfrac{x}{y}+(1-x)\log\tfrac{1-x}{1-y}$ is the Kullback--Leibler divergence. We rewrite
    \[
    nD(k/n\|\eta k/n)=k(\eta-1-\log\eta+\bigO(k/n)).
    \]
    Write $c:=\limsup_{k\to\infty}\lambda_k/k<\eta-1-\log\eta$. We use this to write
    \begin{align*}
        \liminf_{n\gg k\gg 1}\frac{-1}{\lambda_k}\log\P\left(E_{k,n}<\log(n/k)-\log\eta\right)
        &\ge \liminf_{n\gg k\gg 1}\frac{k}{\lambda_k}\left(\eta-1-\log\eta+\bigO(k/n)\right)\\
        &=\frac{\eta-1-\log\eta}{c}>1,
    \end{align*}
    where we interpret the right-hand side as $+\infty$ if $c=0$. This proves the first claim.
    Since $\eta-1-\log\eta\to\infty$ as $\eta\to\infty$, for any $\lambda_k=\bigO(k)$ there exists some $\eta>1$ with $\limsup_k\lambda_k/k<\eta-1-\log\eta$. For $\lambda_k=o(k)$, we have $c=0<\eta-1-\log\eta$ for any $\eta>1$, since $\eta-1-\log\eta>0$ for $\eta>1$.
\end{proof}
\begin{corollary}\label{cor:conditioning}
Let $A_{k,n}$ be any event, let $\lambda_k\to\infty$ with
$\lambda_k=\bigO(k)$, and choose $\eta>1$ so that the conclusion of
\cref{lem:intermediate-ld} holds at speed $\lambda_k$. Define
\[
G_{k,n}:=\{E_{k,n}\ge\log(n/k)-\log\eta\}.
\]
Then
\begin{align*}
\limsup_{n\gg k\gg1}\frac1{\lambda_k}
\log\P(A_{k,n}\mid G_{k,n})\le-1
&\quad\Longrightarrow\quad
\limsup_{n\gg k\gg1}\frac1{\lambda_k}
\log\P(A_{k,n})\le-1,\\
\liminf_{n\gg k\gg1}\frac1{\lambda_k}
\log\P(A_{k,n}\mid G_{k,n})\ge-1
&\quad\Longrightarrow\quad
\liminf_{n\gg k\gg1}\frac1{\lambda_k}
\log\P(A_{k,n})\ge-1.
\end{align*}
In particular, $-\log\P(A_{k,n}\mid G_{k,n})\sim\lambda_k$
implies $-\log\P(A_{k,n})\sim\lambda_k$.
\end{corollary}
\begin{proof}
For the first implication, use
\[
\P(A_{k,n})
\le \P(A_{k,n}\mid G_{k,n})+\P(G_{k,n}^{c}),
\]
and note that the second term has a strictly larger exponential rate.
For the second implication,
\[
\P(A_{k,n})
\ge \P(A_{k,n}\mid G_{k,n})\P(G_{k,n}),
\]
while $\P(G_{k,n})\to1$. Taking logarithms proves both claims.
\end{proof}

\cref{lem:hill-bound} reduces bounding $\alpha H_{k,n}$ to bounding the sum $\sum_{i=1}^{k-1}\Delta_i$ of i.i.d.\ $\text{Exp}(1)$ variables, for which we can directly apply Chernoff bounds and Cramér's theorem: 

\begin{lemma}\label{lem:gamma-cramer}
    Let $\Gamma_m=\sum_{i=1}^m\Delta_i$ for i.i.d.\ $\Delta_i\sim\text{Exp}(1)$. For every $m\ge1$ and $\theta\ge0$,
    \[
    \P(\Gamma_m\ge m(1+\theta))\le e^{-mf(-\theta)},
    \]
    and for every $m\ge1$ and $\theta\in[0,1)$,
    \[
    \P(\Gamma_m\le m(1-\theta))\le e^{-mf(\theta)}.
    \]
    Moreover, these bounds are asymptotically tight:
    \[
    \frac1m\log\P(\Gamma_m\ge m(1+\theta))\to-f(-\theta),\qquad \frac1m\log\P(\Gamma_m\le m(1-\theta))\to-f(\theta).
    \]
\end{lemma}
\begin{proof}
    Since $\Gamma_m\sim\text{Gamma}(m,1)$, $\E[e^{t\Gamma_m}]=(1-t)^{-m}$ for $t<1$. For $\theta\ge0$, the Chernoff bound with $t=\theta/(1+\theta)\in[0,1)$ gives
    \[
    \P(\Gamma_m\ge m(1+\theta))\le e^{-tm(1+\theta)}(1-t)^{-m}=\exp\left(-m\left(\theta-\log(1+\theta)\right)\right)=e^{-mf(-\theta)}.
    \]
    Similarly, for $\theta\in[0,1)$, the Chernoff bound with $s=\theta/(1-\theta)\ge0$ applied to $\E[e^{-s\Gamma_m}]=(1+s)^{-m}$ gives
    \[
    \P(\Gamma_m\le m(1-\theta))\le e^{sm(1-\theta)}(1+s)^{-m}=\exp\left(-m\left(-\theta-\log(1-\theta)\right)\right)=e^{-mf(\theta)}.
    \]
    The limits follow from applying Cramér's theorem to the i.i.d. random variables $(\Delta_i-1)_{i\in[m]}$.
\end{proof}

The tightness claim of \cref{lem:gamma-cramer} holds for fixed $\theta$. \cref{cor:hill-rho-neg} below (and the bias-corrected result in~\cref{sec:md-bias-corrected}) additionally needs a matching lower bound that is \emph{uniform} as $\theta=\theta_m\downarrow0$.

\begin{corollary}\label{cor:gamma-cramer-md}
    Let $\Gamma_m=\sum_{i=1}^m\Delta_i$ for i.i.d.\ $\Delta_i\sim\text{Exp}(1)$. If $\theta_m\downarrow0$ and $m\theta_m^2\to\infty$, then
    \[
    \frac1{m\theta_m^2}\log\P(\Gamma_m\ge m(1+\theta_m))\to-\tfrac12,\qquad \frac1{m\theta_m^2}\log\P(\Gamma_m\le m(1-\theta_m))\to-\tfrac12.
    \]
\end{corollary}
\begin{proof}
    The upper bounds on both probabilities follow from \cref{lem:gamma-cramer} together with $f(\theta)\sim\theta^2/2$ as $\theta\downarrow0$. For the matching lower bounds, fix $s\in\R$ and set $W_m:=(\Gamma_m-m)/(m\theta_m)$. For all large $m$, $|s\theta_m|<1$, and since $\E[e^{t\Gamma_m}]=(1-t)^{-m}$,
    \[
    \frac1{m\theta_m^2}\log\E\left[e^{sm\theta_m^2W_m}\right]=\frac{-m\log(1-s\theta_m)-sm\theta_m}{m\theta_m^2}\longrightarrow\frac{s^2}2,
    \]
    using $-\log(1-s\theta_m)=s\theta_m+\tfrac12s^2\theta_m^2+\bigO(\theta_m^3)$ and $m\theta_m=(m\theta_m^2)/\theta_m\to\infty$. By the Gärtner--Ellis theorem~\citep[Thm.~2.3.6]{dembo1998large}, $W_m$ satisfies a large deviation principle with speed $m\theta_m^2$ and good rate function $x^2/2$. In particular,
    \[
    \liminf_m\frac1{m\theta_m^2}\log\P(W_m>1)\ge-\frac12\quad\text{and}\quad\liminf_m\frac1{m\theta_m^2}\log\P(W_m<-1)\ge-\frac12.
    \]
    Combined with the upper bounds above, this proves the claim.
\end{proof}

We are now ready to prove the upper bounds for the large and moderate deviations:
\begin{lemma}\label{thm:upper-bounds}
    Consider a sequence $\theta_k=\bigO(1)$ with $\theta_k\gg k^{-1/2}$ such that $\theta_k\gg\varepsilon(n/(\eta k))$ for some $\eta>1$. Then
    \[
    \limsup_{n\gg k\gg1}\frac1{kf(-\theta_k)}\log\P(\alpha H_{k,n}>1+\theta_k)\le -1.
    \]
    Moreover, if additionally $\limsup_{k\to\infty}\theta_k<1$, then
    \[
    \limsup_{n\gg k\gg1}\frac1{kf(\theta_k)}\log\P(\alpha H_{k,n}<1-\theta_k)\le -1.
    \]
\end{lemma}
\begin{proof}
    The condition $\theta_k\gg k^{-1/2}$ ensures that $kf(\theta_k),kf(-\theta_k)\to\infty$.
    Write $\varepsilon_k=\varepsilon(n/(\eta k))$. We want to apply \cref{cor:conditioning} with $\lambda_k=kf(\mp\theta_k)=\bigO(k)$, which requires that the conclusion of \cref{lem:intermediate-ld} holds for the given $\eta$ at speed $\lambda_k$. This is not automatic: enlarging $\eta$ makes \cref{lem:intermediate-ld} easier to satisfy, but since $\varepsilon(\cdot)$ is non-increasing, it also weakens the assumption $\theta_k\gg\varepsilon(n/(\eta k))$. We resolve this by splitting the sequence. Fix $\delta\in(0,1)$ with $f(\delta)<\eta-1-\log\eta$ and note that $f(-\theta)\le f(\theta)$ for $\theta\in[0,1)$. Along the indices with $\theta_k\le\delta$, we have $\limsup\lambda_k/k\le f(\delta)<\eta-1-\log\eta$, so \cref{lem:intermediate-ld} holds for the given $\eta$. Along the indices with $\theta_k>\delta$, we have $\theta_k\gg\varepsilon(n/(\eta' k))$ for \emph{every} fixed $\eta'>1$, since $\varepsilon(n/(\eta'k))\to0$ by \cref{lem:potter-rewritten}, so that we may replace $\eta$ by a larger $\eta'$ for which \cref{lem:intermediate-ld} holds at speed $\lambda_k=\bigO(k)$. Since the claims are $\limsup$ statements, it suffices to prove them along each of these two subsequences (whenever infinite). Hence, in the remainder of the proof, we assume without loss of generality that $\eta$ is such that both $\theta_k\gg\varepsilon_k$ and the conclusion of \cref{lem:intermediate-ld} at speed $kf(\mp\theta_k)$ hold, so that \cref{cor:conditioning} lets us condition on $E_{k,n}\ge\log(n/k)-\log\eta$.

    \paragraph{Upper tail.} By~\cref{lem:hill-bound} and~\cref{lem:gamma-cramer},
    \begin{align*}
    \P(\alpha H_{k,n}>1+\theta_k\mid G_{k,n})
    &\le \P\left(\frac1{k}\Gamma_{k-1}> \frac{1+\theta_k-\alpha\log(1+\varepsilon_k)}{1+\varepsilon_k}\right)\\
    &\le \exp\left(-(k-1)f\left(-\theta_k'\right)\right),
    \end{align*}
    where $G_{k,n}=\{E_{k,n}\ge\log(n/k)-\log\eta\}$ is as in \cref{cor:conditioning}, and we used that $\Gamma_{k-1}$ is independent of $G_{k,n}$,
    where the last step applies~\cref{lem:gamma-cramer} with $m=k-1$ and
    \[1+\theta_k'=\tfrac{k}{k-1}\cdot\tfrac{1+\theta_k-\alpha\log(1+\varepsilon_k)}{1+\varepsilon_k}.\]
    Since $\varepsilon_k\ll\theta_k$ by choice of $\eta$, we have $\theta_k'=\theta_k+\bigO(\varepsilon_k)+\bigO(1/k)=\theta_k+o(\theta_k)$. Since $f(\theta_k+o(\theta_k))\sim f(\theta_k)$ holds for both vanishing $\theta_k$ and non-vanishing $\theta_k$, we conclude
    \[
    \limsup_{n\gg k\gg1}\frac1{kf(-\theta_k)}\log\P(\alpha H_{k,n}>1+\theta_k\mid G_{k,n})\le-1,
    \]
    and \cref{cor:conditioning} yields the claim.

    \paragraph{Lower tail.} Symmetrically, using the lower bound of~\cref{lem:hill-bound},
    \begin{align*}
    \P(\alpha H_{k,n}<1-\theta_k\mid G_{k,n})
    &\le \P\left(\frac1{k}\Gamma_{k-1}< \frac{1-\theta_k-\alpha\log(1-\varepsilon_k)}{1-\varepsilon_k}\right)\\
    &\le\exp\left(-(k-1)f(\theta_k'')\right),
    \end{align*}
    where, as above, $1-\theta_k''=\tfrac{k}{k-1}\cdot\tfrac{1-\theta_k-\alpha\log(1-\varepsilon_k)}{1-\varepsilon_k}=1-\theta_k(1+o(1))$, so again $f(\theta_k'')\sim f(\theta_k)$ and the claim follows as above.
\end{proof}


We now prove matching lower bounds for $\P(\alpha H_{k,n}<1-\theta_k)$ and $\P(\alpha H_{k,n}>1+\theta_k)$:

\begin{lemma}\label{thm:lower-bounds}
    Consider a sequence $\theta_k=\bigO(1)$ with $\theta_k\gg k^{-1/2}$ such that $\theta_k\gg\varepsilon(n/(\eta k))$ for some $\eta>1$. Then
    \[
    \liminf_{n\gg k\gg1}\frac1{kf(-\theta_k)}\log\P(\alpha H_{k,n}>1+\theta_k)\ge -1.
    \]
    Moreover, if additionally $\limsup_{k\to\infty}\theta_k<1$, then
    \[
    \liminf_{n\gg k\gg1}\frac1{kf(\theta_k)}\log\P(\alpha H_{k,n}<1-\theta_k)\ge -1.
    \]
\end{lemma}
\begin{proof}
    As in the proof of~\cref{thm:upper-bounds}, fix $\eta>1$ with $\theta_k\gg\varepsilon(n/(\eta k))$ and write $\varepsilon_k=\varepsilon(n/(\eta k))=o(\theta_k)$; by the subsequence argument at the start of the proof of~\cref{thm:upper-bounds}, we may assume that this same $\eta$ satisfies \cref{lem:intermediate-ld} for $\lambda_k=kf(\mp\theta_k)$. Again, \cref{cor:conditioning} allows us to condition on $E_{k,n}\ge\log(n/k)-\log\eta$. By~\cref{lem:hill-bound} and~\cref{lem:gamma-cramer},
    \begin{align*}
    \P(\alpha H_{k,n}>1+\theta_k\mid G_{k,n})
    &\ge \P\left(\frac1k \Gamma_{k-1}> \frac{1+\theta_k-\alpha\log(1-\varepsilon_k)}{1-\varepsilon_k}\right)\\
    &\ge \exp\left(-(k-1)f(-\theta_k')(1+o(1))\right),
    \end{align*}
    where $1+\theta_k'=\tfrac k{k-1}\cdot(1+\theta_k-\alpha\log(1-\varepsilon_k))/(1-\varepsilon_k)=1+\theta_k(1+o(1))$ since $\varepsilon_k=o(\theta_k)$ and $k\to\infty$, so $f(-\theta_k')=f(-\theta_k)(1+o(1))$ as before. (For $\theta_k=\theta$ constant, i.e.\ in the large-deviations regime of~\cref{sec:ld}, this last step is exactly the tightness statement in~\cref{lem:gamma-cramer}. For $\theta_k\downarrow0$, i.e.\ in the moderate-deviations regime of~\cref{sec:md}, it instead follows from the uniform refinement in~\cref{cor:gamma-cramer-md}, since $\theta_k'\downarrow0$ and $kf(-\theta_k')\sim\tfrac12k\theta_k'^2\to\infty$ by $\theta_k\gg k^{-1/2}$. For a general bounded sequence $\theta_k$, it suffices to verify the $\liminf$ along subsequences on which $\theta_k$ converges: if the limit is zero, \cref{cor:gamma-cramer-md} applies; if the limit is $\theta^*>0$, then $\P(\Gamma_{k-1}>(k-1)(1+\theta_k'))\ge\P(\Gamma_{k-1}>(k-1)(1+\theta^*+\epsilon))$ for any $\epsilon>0$ and large $k$, and the tightness statement in~\cref{lem:gamma-cramer} together with the continuity of $f$ and $\epsilon\downarrow0$ gives the claim.) This proves the first claim. The second follows symmetrically using the upper bound of~\cref{lem:hill-bound} and the lower-tail bound of~\cref{lem:gamma-cramer} (respectively~\cref{cor:gamma-cramer-md}).
\end{proof}

Finally, \cref{thm:main-hill} follows directly from \cref{thm:upper-bounds,thm:lower-bounds}.

\section{Large deviations regime (\texorpdfstring{$\lambda_k\asymp k$}{lambda\_k of order k})}\label{sec:ld}
If $\theta_k=\theta>0$ is constant, then $\theta\gg\varepsilon(n/(\eta k))$ holds for every fixed $\eta>1$, since $n/(\eta k)\to\infty$ and \cref{lem:potter-rewritten} gives $\varepsilon(n/(\eta k))\to0$. Hence, \cref{thm:main-hill} yields the following large deviations result:

\begin{corollary}\label{cor:hill-ld}
    For $\theta>0$,
    \[
    -\log\P(\alpha H_{k,n}>1+\theta)\sim kf(-\theta),
    \]
    as $n\gg k\gg 1$. Moreover, if additionally $\theta<1$, then
    \[
    -\log\P(\alpha H_{k,n}<1-\theta)\sim kf(\theta).
    \]
\end{corollary}

\begin{remark}\label{rem:lambert}
The relation $f(\theta)=-\theta-\log(1-\theta)$ can be inverted using the Lambert $W$ function~\citep{Corless_Gonnet_Hare_Jeffrey_Knuth_1996}, which is implicitly defined as the solution of $we^w=x$ for $x\ge -e^{-1}$. There are two solutions, the \emph{principal branch} $W_0(x)\ge-1$ and the secondary branch $W_{-1}(x)\le-1.$
For $\lambda>0$, the $\theta\in(0,1)$ solution of $f(\theta)=\lambda$ is
\[
\theta=1+W_0(-e^{-(1+\lambda)}),
\]
while the negative solution is
\[
\theta=1+W_{-1}(-e^{-(1+\lambda)}).
\]
\end{remark}

\begin{corollary}\label{cor:rescaled-hill-ld}
    The rescaled Hill estimator $\hat\alpha^{(\theta)}_{k,n}=(1-\theta)/H_{k,n}$ from \cref{def:rescaled-hill} achieves the decay rate $\lambda_k=kf(\theta)$.
\end{corollary}

In~\cref{thm:pickands}, we show that the decay rate of a rescaled Pickands estimator $(1-\theta)/P_{k,n}$ is strictly smaller than that of the rescaled Hill estimator as given in~\cref{cor:rescaled-hill-ld}.

In the remainder of this section, we prove an optimality result for the rescaled Hill estimator. For this, we make the following additional assumption on the estimator:

\begin{definition}\label{def:scale-invariant}
    An estimator $\hat\alpha_{k,n}:[0,\infty)^k\to[0,\infty]$ is \emph{scale-invariant} if for all $c>0$ and $(x_1,\dots,x_k)\in[0,\infty)^k$, we have
    \[
    \hat\alpha_{k,n}(cx_1,\dots,cx_k)=\hat\alpha_{k,n}(x_1,\dots,x_k).
    \]
\end{definition}
Moreover, we introduce \emph{normalized regular variation}~\citep{drees1998optimal}:
\begin{definition}\label{def:nrv}
We write $U\in\mathcal{NRV}_{1/\alpha}$ and say that $U$ is \emph{normalized regularly varying} with index $\tfrac1\alpha$ if the derivative $U'(t)$ of $U(t)$ exists for sufficiently large $t$ and $tU'(t)/U(t)\to\frac1\alpha$ as $t\to\infty$.
\end{definition}

We prove the following result:

\begin{theorem}\label{thm:hill-optimality}
    Consider any scale-invariant estimator $\hat\alpha_{k,n}:[0,\infty)^k\to[0,\infty]$ that achieves
    \[
    \liminf_{n\gg k\gg 1}\P\left(\hat\alpha_{k,n}(X_{1:k,n})>\alpha(1-\theta)\right)>0,
    \]
    for some $\theta\in(0,1),\alpha>0$ and $X_1,\dots,X_n\sim\text{RV}(U)$ with $U\in\mathcal{NRV}_{1/\alpha}$.
    Then the decay rate of $\hat\alpha_{k,n}$ is at most $k\cdot f(\theta)+o(k)$,
    which matches the rate of the conservative Hill estimator $\hat\alpha^{(\theta)}_{k,n}$ up to $o(k)$.
\end{theorem}
\begin{proof}
    Define $Y_i=E_{i,n}-E_{k,n}$ for $i\in[k-1]$ and \[q_{k,n}(y_{1:k-1};t)=\hat\alpha_{k,n}(U(e^{t+y_1}),\dots,U(e^{t+y_{k-1}}),U(e^t)).\] 
    By assumption, we have
    \[
    \liminf_{n\gg k\gg 1}\P\left(q_{k,n}(Y_{1:k-1};E_{k,n})>(1-\theta)\alpha\right)>0.
    \]
    Moreover, notice that $\sum_{i=1}^{k-1}Y_i=\sum_{i=1}^{k-1}(E_{i,n}-E_{k,n})=\sum_{i=1}^{k-1}\sum_{j=i}^{k-1}\Delta_j/j=\sum_{j=1}^{k-1}\Delta_j=\Gamma_{k-1}$, so that by~\cref{lem:gamma-cramer}, 
    \[
        \P\left(\Gamma_{k-1}< k-k^{2/3}\right)\to0.
    \]
    Let us define the set
    \[
        \mathcal A_{k,n}(t)=\left\{y_1\ge\dots\ge y_{k-1}\ge0\ :\ q_{k,n}(y_{1:k-1};t)>(1-\theta)\alpha, \sum_{i=1}^{k-1}y_i\ge k-k^{2/3}\right\}.
    \]
    Then,
    \begin{align*}
    \P\left(Y_{1:k-1}\in\mathcal A_{k,n}(E_{k,n})\right)\ge{}& \P\left(q_{k,n}(Y_{1:k-1};E_{k,n})>(1-\theta)\alpha\right)\\
    &-\P\left(\Gamma_{k-1}< k-k^{2/3}\right),
    \end{align*}
    so that
    \[
    \liminf_{n\gg k\gg 1}\P\left(Y_{1:k-1}\in\mathcal A_{k,n}(E_{k,n})\right)>0.
    \]
    Next, define $X_{i,n}'=U(e^{E_{i,n}/(1-\theta)})$ and notice that $X_{i,n}'$ correspond to order statistics of a regularly varying distribution with exponent $(1-\theta)\alpha$. We will show that the decay rate of $\hat\alpha_{k,n}$ under this distribution is at most $k\cdot f(\theta)+o(k)$.

    Define $\rho(y;t)$ implicitly by
    \begin{equation}\label{eq:rho-def}
\frac{U\left(e^{(t+\rho(y;t))/(1-\theta)}\right)}{U\left(e^{t/(1-\theta)}\right)}=\frac{U(e^{t+y})}{U(e^t)}.
\end{equation}
    Since $U$ is nondecreasing with derivative $sU'(s)/U(s)\to\frac1\alpha>0$ due to normalized regular variation, there exists some $t_0$ so that $U'(s)>0$ for all $s\ge e^{t_0}$.
    Hence, for $t\ge t_0$, both $\rho\mapsto U\left(e^{(t+\rho)/(1-\theta)}\right)/U\left(e^{t/(1-\theta)}\right)$ and $z\mapsto
U\left(e^{t+z}\right)/U\left(e^{t}\right)$ are bijections from $[0,\infty)$ to $[1,\infty)$, so that $y\mapsto\rho(y;t)$ is a bijection from $[0,\infty)$ to $[0,\infty)$.
Define $\psi(t)=\log U(e^t)$, which is strictly increasing and eventually differentiable  with derivative $\psi'(t)=e^tU'(e^t)/U(e^t)\to\frac1\alpha$ due to the normalized regular variation of $U$.
\eqref{eq:rho-def} can be rewritten to
\begin{equation}\label{eq:rho-def-psi}
\psi\left(\frac{t+\rho(y;t)}{1-\theta}\right)-\psi\left(\frac{t}{1-\theta}\right)=\psi(t+y)-\psi(t).
\end{equation}
By taking the derivative of both sides of~\eqref{eq:rho-def-psi} with respect to $y$ for $t\ge t_0$, we find
\begin{equation}\label{eq:rho-derivative}
\rho'(y;t):=\frac{\partial\rho}{\partial y}(y;t)=(1-\theta)\frac{\psi'(t+y)}{\psi'\left(\frac{t+\rho(y;t)}{1-\theta}\right)}.
\end{equation}
Now, define
\[
\mathcal A_{k,n}'(t)=\left\{(\rho(y_1;t),\dots,\rho(y_{k-1};t))\ :\ (y_1,\dots,y_{k-1})\in\mathcal A_{k,n}(t)\right\}.
\]
By scale-invariance of $\hat\alpha_{k,n}$, we have
\begin{align*}
    &\hat\alpha_{k,n}\left(U(e^{(t+\rho(y_1;t))/(1-\theta)}),\dots,U(e^{(t+\rho(y_{k-1};t))/(1-\theta)}),U(e^{t/(1-\theta)})\right)\\
    &=\hat\alpha_{k,n}\left(\frac{U(e^{(t+\rho(y_1;t))/(1-\theta)})}{U(e^{t/(1-\theta)})},\dots,\frac{U(e^{(t+\rho(y_{k-1};t))/(1-\theta)})}{U(e^{t/(1-\theta)})},1\right)\\
    &=\hat\alpha_{k,n}\left(\frac{U(e^{t+y_1})}{U(e^t)},\dots,\frac{U(e^{t+y_{k-1}})}{U(e^t)},1\right)\\
    &=\hat\alpha_{k,n}\left(U(e^{t+y_1}),\dots,U(e^{t+y_{k-1}}),U(e^t)\right).
\end{align*}
Hence, $Y_{1:k-1}\in\mathcal A_{k,n}'(E_{k,n})$ implies $\hat\alpha_{k,n}(X_{1:k,n}')>(1-\theta)\alpha$, so that
\[
\P(\hat\alpha_{k,n}(X_{1:k,n}')>(1-\theta)\alpha)\ge\P(Y_{1:k-1}\in\mathcal A_{k,n}'(E_{k,n})).
\]
Notice that the density of $Y_{1:k-1}$ is
\[
f_Y(y_{1:k-1})=(k-1)!\exp\left(-\sum_{i=1}^{k-1}y_i\right),
\]
for $y_1\ge\dots\ge y_{k-1}\ge0$.
Using a coordinate transformation from $\mathcal A_{k,n}'(E_{k,n})$ to $\mathcal A_{k,n}(E_{k,n})$ via $\rho(y_i;E_{k,n})\mapsto y_i$, we can write
\begin{align*}
    &\P(Y_{1:k-1}\in\mathcal A_{k,n}'(t)\mid E_{k,n}=t)\\
    =&\int_{\mathcal A_{k,n}'(t)}f_Y(y'_{1:k-1})\,dy'_1\dots dy'_{k-1}\\
    =&\int_{\mathcal A_{k,n}(t)}f_Y(\rho(y_1;t),\dots,\rho(y_{k-1};t))\,d\rho(y_1;t)\dots d\rho(y_{k-1};t)\\
    =&\int_{\mathcal A_{k,n}(t)}f_Y(\rho(y_1;t),\dots,\rho(y_{k-1};t))\,\prod_{i=1}^{k-1}\rho'(y_i;t)\,dy_1\dots dy_{k-1}\\
    =&(k-1)!\int_{\mathcal A_{k,n}(t)}\prod_{i=1}^{k-1}e^{-\rho(y_i;t)}\rho'(y_i;t)\,dy_1\dots dy_{k-1}.
\end{align*}
Since $\psi'(t)\to\frac1\alpha$ as $t\to\infty$, the quantity $\varepsilon'(t):=\sup_{s\ge t}|\alpha\psi'(s)-1|$ satisfies $\varepsilon'(t)\to0$ as $t\to\infty$.
From~\eqref{eq:rho-derivative}, we get
\[
(1-\theta)\frac{1-\varepsilon'}{1+\varepsilon'}\le \rho'(y;t)\le (1-\theta)\frac{1+\varepsilon'}{1-\varepsilon'},
\]
for $\varepsilon'=\varepsilon'(t)$. This leads to
\[
    \rho(y;t)=\int_0^y\rho'(z;t)dz\le (1-\theta)\frac{1+\varepsilon'}{1-\varepsilon'}y.
\]
Together, this yields
  \begin{align*}
  e^{-\rho(y;t)}\rho'(y;t)
  &\ge (1-\theta)\frac{1-\varepsilon'}{1+\varepsilon'}\exp\left(-(1-\theta)\frac{1+\varepsilon'}{1-\varepsilon'}y\right)\\
  &=\exp\left(-y+\theta y+\log(1-\theta)+\log\frac{1-\varepsilon'}{1+\varepsilon'}-\frac{2(1-\theta)\varepsilon' y}{1-\varepsilon'}\right).
  \end{align*}
  The $-y$ in the exponent will be absorbed in $f_Y$. 
  To lower-bound the other occurrences of $y$, we use that $y_{1:k-1}\in\mathcal A_{k,n}(t)$ satisfies $\sum_{i=1}^{k-1}y_i\ge k-k^{2/3}$, and consider $t$ large enough so that
  \begin{equation}\label{eq:eps-prime-small}
    \theta-\frac{2(1-\theta)\varepsilon'(t)}{1-\varepsilon'(t)}>0.
  \end{equation}
  This allows us to bound
  \begin{align*}
    \prod_{i=1}^{k-1}e^{-\rho(y_i;t)}\rho'(y_i;t)
    \ge e^{c_k(\theta,\varepsilon')-\sum_{i=1}^{k-1}y_i},
  \end{align*}
  where
  \[
    c_k(\theta,\varepsilon')=(k-1)\left(\log(1-\theta)+\log\frac{1-\varepsilon'}{1+\varepsilon'}\right)+\left(\theta-\frac{2(1-\theta)\varepsilon'}{1-\varepsilon'}\right)(k-k^{2/3}).
  \]
  This leads to
  \begin{align*}
  \P(Y_{1:k-1}\in\mathcal A_{k,n}'(t)\mid E_{k,n}=t)
  &\ge e^{c_k(\theta,\varepsilon'(t))}\int_{\mathcal A_{k,n}(t)}f_Y(y_{1:k-1})\,dy_1\dots dy_{k-1}\\
  &=e^{c_k(\theta,\varepsilon'(t))}\P(Y_{1:k-1}\in\mathcal A_{k,n}(t)).
  \end{align*}
  Define $t_k=\log(n/k)-1$ so that $\P(E_{k,n}<t_k)\to0$ and~\eqref{eq:eps-prime-small} holds eventually. We can write
  \begin{align}
    &\P(\hat\alpha_{k,n}(X_{1:k,n}')>(1-\theta)\alpha)\nonumber\\
    \ge &\P(Y_{1:k-1}\in\mathcal A_{k,n}'(E_{k,n}),E_{k,n}\ge t_k)\nonumber\\
    =&\P(E_{k,n}\ge t_k)\,\E\left[\P(Y_{1:k-1}\in\mathcal A_{k,n}'(E_{k,n})\mid E_{k,n}) \middle|\ E_{k,n}\ge t_k\right]\nonumber\\
    \ge&\P(E_{k,n}\ge t_k)\,\E\left[e^{c_k(\theta,\varepsilon'(E_{k,n}))}\P(Y_{1:k-1}\in\mathcal A_{k,n}(E_{k,n})) \middle|\ E_{k,n}\ge t_k\right]\nonumber\\
    \ge &\P(E_{k,n}\ge t_k)\,e^{c_k(\theta,\varepsilon'(t_k))}\P(Y_{1:k-1}\in\mathcal A_{k,n}(E_{k,n}) \mid E_{k,n}\ge t_k)\label{eq:ckn-dec}\\
    \ge &\P(E_{k,n}\ge t_k)\,e^{c_k(\theta,\varepsilon'(t_k))}\left(\P(Y_{1:k-1}\in\mathcal A_{k,n}(E_{k,n}))-\P(E_{k,n}<t_k)\right)\nonumber.
  \end{align}
Here, the second inequality uses the lower bound on $\P(Y_{1:k-1}\in\mathcal A_{k,n}'(t)\mid E_{k,n}=t)$ derived above, which is valid for $t\ge t_k$ since~\eqref{eq:eps-prime-small} holds for such $t$, and~\eqref{eq:ckn-dec} follows from the fact that $c_k(\theta,\varepsilon')$ is decreasing in $\varepsilon'$, and $\varepsilon'(E_{k,n})\le \varepsilon'(t_k)$ for $E_{k,n}\ge t_k$.
Since $\varepsilon'(t_k)\to0$ as $k\to\infty$, we have $c_k(\theta,\varepsilon'(t_k))= -kf(\theta)+o(k)$. Moreover, since $\P(Y_{1:k-1}\in\mathcal A_{k,n}(E_{k,n}))$ is bounded away from zero and $\P(E_{k,n}<t_k)\to0$, the last factor is bounded away from zero, as is $\P(E_{k,n}\ge t_k)\to1$, so that their product equals $e^{o(k)}$. We conclude that
\[
  \P(\hat\alpha_{k,n}(X_{1:k,n}')>(1-\theta)\alpha)\ge \exp\left(-kf(\theta)+o(k)\right),
\]
which is exactly the statement that the decay rate of $\hat\alpha_{k,n}$ is at most $kf(\theta)+o(k)$.
\end{proof}
Notice that \cref{thm:hill-optimality} does not restrict \cref{def:conservative} to normalized regularly varying distributions, but rather shows that if $\hat\alpha_{k,n}$ achieves a bias level $\theta$ for any NRV distribution, then the decay rate over all regularly varying distributions is at most $kf(\theta)+o(k)$.

In the proof of~\cref{thm:hill-optimality}, we used normalized regular variation to ensure that $q_{k,n}(y_1,\dots,y_{k-1};t)$ is sufficiently smooth for large $t$, which allowed us to lower-bound $\P(\hat\alpha_{k,n}(X_{1:k,n}')>(1-\theta)\alpha)$ in terms of $\P(\hat\alpha_{k,n}(X_{1:k,n})>(1-\theta)\alpha)$. Instead of assuming normalized regular variation, it may also be possible to achieve this same smoothness of $q_{k,n}(y_1,\dots,y_{k-1};t)$ by imposing additional smoothness conditions on $\hat\alpha_{k,n}$.

The next example shows that the optimality result of~\cref{thm:hill-optimality} does not hold for estimators that are not scale-invariant. 
\begin{lemma}\label{lem:non-scale-invariant}
    Let $0<\varepsilon<\theta<1$. Consider the estimator
\begin{equation}\label{eq:non-scale-invariant}
    \hat\beta_{k,n}=(1-\varepsilon)\frac{\log(n/k)}{\log X_{k,n}}.
\end{equation} 
Then $\P(\hat\beta_{k,n}(X_{1:k,n})>(1-\theta)\alpha)\to1$, while
\[
-\log\P(\hat\beta_{k,n}(X_{1:k,n})>\alpha)\gg kf(\theta).
\]
\end{lemma}
\begin{proof}
    The first claim follows from $\hat\beta_{k,n}\stackrel\P\to(1-\varepsilon)\alpha$, as $\log X_{k,n}\stackrel\P\sim\log U(n/k)\sim\alpha^{-1}\log(n/k)$.
    For the second claim, write $L_n:=\log(n/k)$ and fix any $c\in(1-\varepsilon,1)$. Since $U\in\mathcal{RV}'_{1/\alpha}$, we have $\log U(t)\sim\alpha^{-1}\log t$ as $t\to\infty$, so that
    \[
    \log U\left(e^{cL_n}\right)\sim\frac c\alpha L_n>\frac{1-\varepsilon}\alpha L_n
    \]
    for all sufficiently large $n,k$. Since $X_{k,n}=U(e^{E_{k,n}})$ and $U$ is non-decreasing, $E_{k,n}\ge cL_n$ implies $\log X_{k,n}\ge\log U(e^{cL_n})>\frac{1-\varepsilon}\alpha L_n$ for all sufficiently large $n,k$. Consequently,
    \[
    \{\hat\beta_{k,n}>\alpha\}=\left\{\log X_{k,n}<\frac{1-\varepsilon}\alpha L_n\right\}\subseteq\{E_{k,n}<cL_n\},
    \]
    where the equality uses the convention $c/T:=+\infty$ on $\{T\le0\}$ for reciprocal estimators.
    The event $\{E_{k,n}<cL_n\}$ means that fewer than $k$ out of $n$ independent standard exponentials exceed $cL_n$, which happens with probability $p_n:=e^{-cL_n}=(k/n)^c$ each. Hence, by the Chernoff--Hoeffding inequality,
    \[
    \P(\hat\beta_{k,n}>\alpha)\le\P\left(\text{Bin}(n,p_n)<k\right)\le\exp\left(-nD\left(\frac kn\middle\|p_n\right)\right),
    \]
    where $D(a\|b)=a\log(a/b)+(1-a)\log((1-a)/(1-b))$ is the Kullback--Leibler divergence. Since $c<1$, we have $k/n\ll p_n\ll1$, and in this regime $D(k/n\|p_n)\sim p_n$. Therefore,
    \[
    -\log\P(\hat\beta_{k,n}>\alpha)\ge nD\left(\frac kn\middle\|p_n\right)\sim np_n=k(n/k)^{1-c}\gg k,
    \]
    which indeed grows at a faster rate than $kf(\theta)$ for any fixed $\theta\in(\varepsilon,1)$.
\end{proof}

The following result is a corollary to~\cref{thm:hill-optimality} and highlights that within the class of scale-invariant estimators, a decay rate $\lambda_k=\Theta(k)$ comes at the cost of consistency.
\begin{corollary}\label{cor:conservativeness-consistency}
    Consider a scale-invariant estimator $\hat\alpha_{k,n}$ that is a function of the $k$ largest observations $X_{1:k,n}$ from a normalized regularly-varying distribution. If $\hat\alpha_{k,n}$ achieves a decay rate $\lambda_k=\Theta(k),$ then $\hat\alpha_{k,n}$ cannot be a consistent estimator for $\alpha$.
\end{corollary}
\begin{proof}
    By assumption, there exists some $\theta>0$ such that $\liminf_{n\gg k\gg1}\lambda_k/k>f(\theta)$. If $\hat\alpha_{k,n}$ is consistent, then $\P(\hat\alpha_{k,n}(X_{1:k,n})>\alpha(1-\theta))\to1$ as $n\gg k\gg1$. But by~\cref{thm:hill-optimality}, this implies that the decay rate is at most $kf(\theta)+o(k)$, which contradicts the assumption that the decay rate is at least $\lambda_k$.
\end{proof}
\cref{cor:conservativeness-consistency} additionally tells us that a consistent scale-invariant estimator that is conservative with decay rate $1\ll \lambda_n\ll n$ needs to use at least $k\gg\lambda_n$ order statistics.

\section{Moderate deviations regime (\texorpdfstring{$\lambda_k\ll k$}{lambda\_k much smaller than k})}\label{sec:md}
To get a consistent estimator, we need $\theta_k\downarrow0$, which corresponds to the moderate deviations regime $\lambda_k\ll k$.
\cref{thm:main-hill} allows for $\theta_k\to0$ as long as $\theta_k\gg\varepsilon(n/(\eta k))$ for some $\eta>1$. In this section, we will show what speeds of $\theta_k$ are admissible for the moderate deviations regime, depending on the second-order condition of the underlying distribution.

Recall that the second-order condition~\eqref{eq:soc} comes with a second-order parameter $\rho\le0$.
\eqref{eq:hill-normality} suggests that $\alpha H_{k,n}-1=\bigO_\P(|A(n/k)|+k^{-1/2})$, so that moderate deviations require $\theta_k\gg |A(n/k)|+k^{-1/2}$.
We distinguish the cases $\rho<0$ and $\rho=0$. The case $\rho<0$ allows us to bound $\varepsilon(t)\le 2|\rho|^{-1}\sup_{t'\ge t}|A_0(t')|$, where $|A_0(t)|\sim|A(t)|$. This indeed gives us the desired moderate deviations result for $\theta_k\gg |A(n/k)|+k^{-1/2}$. The case $\rho=0$ is more challenging and requires a different approach. 

Our results are valid even when $\sqrt k |A(n/k)|\to\infty$, in contrast to the CLT result from Theorem~3.2.5 of~\citet{haan2006extreme}.
We additionally derive a moderate deviations result for a bias-corrected Hill estimator for $\rho<0$, which allows for deviations of size $\theta_k\asymp|A(n/k)|$ if $\sqrt k|A(n/k)|\to\infty$.

When deriving decay rates in the moderate deviations regime, we restrict~\cref{def:conservative} to second-order regularly varying distributions with a given rate function $A(t)$. That is, we say that $\hat\alpha_{k,n}$ is conservative with decay rate $\lambda_n$ among distributions with second-order rate $A(t)$ if~\eqref{eq:conservative-single} holds for any $\alpha>0$, $U\in2\mathcal{RV}_{1/\alpha,A}$ and $X_1,\dots,X_n\sim \text{RV}(U)$.

\subsection{The case \texorpdfstring{$\rho<0$}{rho<0}}

\begin{lemma}\label{lem:potter-soc}
    Suppose $U$ satisfies \eqref{eq:soc} for some $A(t)$ with $\rho<0$, then $\varepsilon(t)=\bigO(|A(t)|)$.
\end{lemma}
\begin{proof}
    We rely on Theorem~2.3.9 of~\citet{haan2006extreme}, which states that if $U$ satisfies \eqref{eq:soc}, then for any $\varepsilon,\delta>0$, there exist $t_0=t_0(\varepsilon,\delta)>1$ such that for all $t,tx\ge t_0$,
    \begin{equation}\label{eq:haan-soc}
    \left|\frac{\frac{U(tx)}{U(t)}-x^{\frac1\alpha}}{A_0(t)}-x^{\frac1\alpha}\int_1^xu^{\rho-1}du\right|\le\varepsilon x^{\frac1\alpha+\rho}\max\{x^\delta,x^{-\delta}\},
    \end{equation}
    for some $A_0(t)\to0$ that is regularly varying\footnote{\citet{haan2006extreme} gives an explicit expression for $A_0(t)$, which is not needed in this proof.} with index $\rho\le 0$.
    For $\rho<0$, we can pick $\delta=-\rho/2>0$, $\varepsilon=-\rho^{-1}>0$ and $x=e^y\ge1$. Then, $\int_1^{e^y}u^{\rho-1}du=\frac{e^{\rho y}-1}\rho\ge0$ and $\max\{x^{\delta},x^{-\delta}\}=e^{-\rho y/2}$.
    We rewrite 
    \[
    \left|\frac{\frac{U(tx)}{U(t)}-x^{\frac1\alpha}}{A_0(t)}-x^{\frac1\alpha}\int_1^xu^{\rho-1}du\right|=\frac{\left|\frac{U(tx)}{U(t)}-x^{\frac1\alpha}-A_0(t)x^{\frac1\alpha}\frac{e^{\rho y}-1}\rho\right|}{|A_0(t)|},
    \]
    and multiply both sides of~\eqref{eq:haan-soc} by $|A_0(t)|$ and rearrange to get
\begin{align*}
A_0(t)e^{y/\alpha}\frac{e^{\rho y}-1}\rho-|A_0(t)|\varepsilon e^{y/\alpha}e^{\rho y/2}\ &\le\ \frac{U(te^y)}{U(t)}-e^{y/\alpha}\\
&\le\ A_0(t)e^{y/\alpha}\frac{e^{\rho y}-1}\rho+|A_0(t)|\varepsilon e^{y/\alpha}e^{\rho y/2},
\end{align*}
for $t\ge t_0$.
Taking the worst case over $|A_0(t)|\le |A_0(t)|$ ($A_0(t)=-|A_0(t)|$ for the lower bound, $A_0(t)=|A_0(t)|$ for the upper bound) gives
\begin{equation}\label{eq:haan-soc-rewritten2}
\left(1-|A_0(t)|\beta(y)\right)e^{y/\alpha}\ \le\ \frac{U(te^y)}{U(t)}\ \le\ \left(1+|A_0(t)|\beta(y)\right)e^{y/\alpha},
\end{equation}
for $\beta(y):=\frac{e^{\rho y}-1}\rho-\rho^{-1}e^{\rho y/2}$.
Writing $w=e^{\rho y/2}\in(0,1]$, we have $\beta(y)=-\rho^{-1}(1+w-w^2)$. Since $1+w-w^2\in[1,\tfrac54]$ for $w\in(0,1]$, this gives $\beta(y)\in\left[-\rho^{-1},-\tfrac54\rho^{-1}\right]\subset(0,-2\rho^{-1})$, so $\beta(y)$ is bounded and positive, uniformly in $y\ge0$. Since $|A_0(t)|\to0$, there exists a $t_0'$ so that $-2\rho^{-1}|A_0(t)|<1$ for all $t\ge t_0'$. Combining this with~\eqref{eq:haan-soc-rewritten2} and $\beta(y)<-2\rho^{-1}$, we obtain, for all $y\ge0$ and $t\ge\max\{t_0,t_0'\}$,
\[
\left(1+2\rho^{-1}|A_0(t)|\right)e^{y/\alpha}\le\frac{U(te^y)}{U(t)}\le\left(1-2\rho^{-1}|A_0(t)|\right)e^{y/\alpha}.
\]
By additionally bounding
\[
\left(1-2\rho^{-1}|A_0(t)|\right)e^{y/\alpha}\le \left(1-2\rho^{-1}|A_0(t)|\right)e^{(1-2\rho^{-1}|A_0(t)|)y/\alpha},
\]
and
\[
\left(1+2\rho^{-1}|A_0(t)|\right)e^{y/\alpha}\ge \left(1+2\rho^{-1}|A_0(t)|\right)e^{(1+2\rho^{-1}|A_0(t)|)y/\alpha},
\]
and noticing that these bounds hold for all $y\ge0$ and $t\ge\max\{t_0,t_0'\}$, we conclude that $\varepsilon(t)\le-2\rho^{-1}\sup_{t'\ge t}|A_0(t')|$.
Since $A_0(t)\sim A(t)$ and $|A|\in\mathcal{RV}_\rho$ with $\rho<0$, the uniform convergence theorem for regularly varying functions with negative index gives $\sup_{t'\ge t}|A_0(t')|\sim |A(t)|$, so that $\varepsilon(t)=\bigO(|A(t)|)$.
\end{proof}
\cref{lem:potter-soc} allows us to replace the condition $\theta_k\gg\varepsilon(n/(\eta k))$ in~\cref{thm:main-hill} with the more explicit condition $\theta_k\gg |A(n/k)|$ for $\rho<0$: fixing any $\eta>1$, \cref{lem:potter-soc} gives $\varepsilon(n/(\eta k))=\bigO(|A(n/(\eta k))|)$, and since $A$ is regularly varying with index $\rho$, $|A(n/(\eta k))|=\Theta(|A(n/k)|)$. This gives the following corollary:

\begin{corollary}\label{cor:hill-rho-neg}
    If $U$ satisfies~\eqref{eq:soc} with rate $A(t)$ and $\rho<0$, then for any sequence $1\gg\theta_k\gg |A(n/k)|+k^{-1/2}$, we have
    \[
        -\log\P(\alpha H_{k,n}>1+\theta_k)\sim \frac{k}{2}\theta_k^2\qquad\text{and}\qquad -\log\P(\alpha H_{k,n}<1-\theta_k)\sim \frac{k}{2}\theta_k^2,
    \]
    as $n\gg k\gg 1$.
\end{corollary}

\subsection{The bias-corrected Hill estimator}\label{sec:md-bias-corrected}
\cref{cor:hill-rho-neg} requires $\theta_k$ to asymptotically \emph{dominate} the bias term $A(n/k)$. This is because \cref{lem:potter-soc} only controls the \emph{magnitude} $\varepsilon(t)=\bigO(|A(t)|)$ of the second-order remainder, discarding its sign. 
The second-order rate $A(n/k)$ specifies the leading term of the bias of the Hill estimator.
In this section, we provide the moderate deviations of the \emph{bias-corrected Hill estimator}
\begin{equation}\label{eq:bias-corrected-hill}
\tilde H_{k,n} := H_{k,n} - \frac{A(n/k)}{1-\rho},
\end{equation}
which allows $\theta_k$ to be of the same order as $A(n/k)$.
Write $q_\rho(y):=\frac{e^{\rho y}-1}\rho=\int_1^{e^y}u^{\rho-1}du$ for $y\ge0$. Since $\rho<0$,
\begin{equation}\label{eq:qrho-bounds}
0\le q_\rho(y)\le\frac1{|\rho|},\qquad \E[q_\rho(\Delta)]=\frac1{1-\rho}=:\mu_\rho,\qquad \Delta\sim\text{Exp}(1),
\end{equation}
where the mean identity follows from $\E[e^{\rho\Delta}]=(1-\rho)^{-1}$.

We first extract a \emph{signed} refinement of the uniform second-order theorem, keeping track of the sign of $A_0(t)$ rather than discarding it as in the proof of \cref{lem:potter-soc}.

\begin{lemma}\label{lem:signed-soc}
    Suppose $U$ satisfies~\eqref{eq:soc} with $\rho<0$, and let $A_0$ be as in~\eqref{eq:haan-soc} (and Theorem~2.3.9 of~\citet{haan2006extreme}). For every $\varepsilon\in(0,1]$ there is some $t_0(\varepsilon)<\infty$ such that, for all $t\ge t_0(\varepsilon)$ and $y\ge0$,
    \[
    \left|\alpha\log\frac{U(te^y)}{U(t)}-y-\alpha A_0(t)q_\rho(y)\right|\le \alpha\varepsilon|A_0(t)|+\alpha(1+|\rho|^{-1})^2A_0(t)^2.
    \]
\end{lemma}
\begin{proof}
    Fix $\delta:=-\rho/2>0$ and apply~\eqref{eq:haan-soc} with $x=e^y\ge1$. Since $\int_1^{e^y}u^{\rho-1}du=q_\rho(y)$ and, because $\delta>0$ and $x\ge1$, $\max\{x^\delta,x^{-\delta}\}=e^{\delta y}$, multiplying with $|A_0(t)|$ gives, for $t\ge t_0(\varepsilon,\delta)$,
    \[
    \left|\frac{U(te^y)}{U(t)}-e^{y/\alpha}-A_0(t)e^{y/\alpha}q_\rho(y)\right|\le\varepsilon|A_0(t)|e^{y/\alpha}e^{(\rho+\delta)y}\le\varepsilon|A_0(t)|e^{y/\alpha},
    \]
    where we used $\rho+\delta=\rho/2<0$ and $y\ge0$. Divide by $e^{y/\alpha}$ and write $v:=U(te^y)e^{-y/\alpha}/U(t)-1$. This gives
    \[
    |v-A_0(t)q_\rho(y)|\le\varepsilon|A_0(t)|.
    \]
    Combined with~\eqref{eq:qrho-bounds} and $\varepsilon\le1$, this gives
    \[
    |v|\le|A_0(t)|(|\rho|^{-1}+\varepsilon)\le(1+|\rho|^{-1})|A_0(t)|.
    \]
    Since $A_0(t)\to0$, there is $t_0'<\infty$ such that $(1+|\rho|^{-1})|A_0(t)|\le\tfrac12$ for all $t\ge t_0'$. Set $t_0(\varepsilon):=\max\{t_0',t_0(\varepsilon,\delta)\}$, so that $|v|\le\tfrac12$ for $t\ge t_0(\varepsilon)$. Then $|\log(1+v)-v|\le v^2$ gives
    \[
    |\alpha\log(1+v)-\alpha v|\le\alpha v^2\le\alpha(1+|\rho|^{-1})^2A_0(t)^2.
    \]
    Combining this with $|v-A_0(t)q_\rho(y)|\le\varepsilon|A_0(t)|$ via the triangle inequality,
    \begin{align*}
    \left|\alpha\log\frac{U(te^y)}{U(t)}-y-\alpha A_0(t)q_\rho(y)\right|
    &=\left|\alpha\log(1+v)-\alpha A_0(t)q_\rho(y)\right|\\
    &\le\alpha\varepsilon|A_0(t)|+\alpha(1+|\rho|^{-1})^2A_0(t)^2,
    \end{align*}
    uniformly over $t\ge t_0(\varepsilon)$ and $y\ge0$, as required.
\end{proof}

We also need $E_{k,n}$ to concentrate at a scale finer than the fixed-$\eta$ scale used in \cref{lem:intermediate-ld}:

\begin{lemma}\label{lem:ekn-concentration}
    For all sufficiently large $k$, all $n\ge k$, and all $r\in[4/k,1)$,
    \[
    \P\left(\left|E_{k,n}-\log\frac nk\right|>r\right)\le2\exp(-kr^2/32).
    \]
\end{lemma}
\begin{proof}
    Recall $E_{k,n}=\sum_{j=k}^n\Delta_j/j$, so $\E E_{k,n}=\sum_{j=k}^nj^{-1}$. Comparing with $\int_k^ns^{-1}ds$ gives $|\E E_{k,n}-\log(n/k)|\le2/k$, so for $r\ge4/k$ it suffices to bound $\P(|W_{k,n}|>r/2)$ for
    \[
    W_{k,n}:=E_{k,n}-\E E_{k,n}=\sum_{j=k}^n\frac{\Delta_j-1}j.
    \]
    Write $\sigma_k^2:=\sum_{j=k}^nj^{-2}\le2/k$, where the last inequality follows from an integral bound. For $0\le\lambda\le k/2$, the bound $-u-\log(1-u)\le u^2$ on $[0,\tfrac12]$ gives
    \[
    \log\E e^{\lambda W_{k,n}}\le\lambda^2\sigma_k^2\le\frac{2\lambda^2}k,
    \]
    and taking $\lambda=kr/8$, a Chernoff bound gives
    \[
    \P(W_{k,n}>r/2)\le\exp\left(-\frac{kr^2}{32}\right).
    \]
    Symmetrically, $u-\log(1+u)\le u^2/2$ for $u\ge0$ gives $\log\E e^{-\lambda W_{k,n}}\le\lambda^2\sigma_k^2/2\le\lambda^2/k$, and the same $\lambda$ gives
    \[
    \P(W_{k,n}<-r/2)\le\exp\left(-\frac{3kr^2}{64}\right)\le\exp\left(-\frac{kr^2}{32}\right).
    \]
    Combining the two tails via a union bound gives $\P(|W_{k,n}|>r/2)\le2\exp(-kr^2/32)$, which proves the claim.
\end{proof}

Finally, the correction term concentrates around its mean at the fast rate $\Theta(k)$, uniformly and regardless of $A(n/k)$:

\begin{lemma}\label{lem:correction-concentration}
    Let $Q_k:=\frac1k\sum_{i=1}^{k-1}q_\rho(E_{i,n}-E_{k,n})$. For every $\tau\in(0,1)$ and every $k\ge\max\{2,2\mu_\rho/\tau\}$,
    \[
    \P(|Q_k-\mu_\rho|>\tau)\le2\exp(-\rho^2\tau^2k/4).
    \]
\end{lemma}
\begin{proof}
    As in the proof of \cref{lem:hill-bound}, we note that by the Rényi representation~\citep{renyi1953theory}, the (multi-)set $\{\sum_{j=i}^{k-1}\Delta_j/j\}_{i\in[k-1]}$ is equal in distribution to $k-1$ i.i.d.\ $\text{Exp}(1)$ variables $\{Y_1,\dots,Y_{k-1}\}$. Since $g(z_1,\dots,z_{k-1}):=\frac1k\sum_iq_\rho(z_i)$ is a symmetric function of its $k-1$ arguments, it therefore holds that
    \[
    Q_k\stackrel d=\frac1k\sum_{i=1}^{k-1}q_\rho(Y_i).
    \]
    The summands $q_\rho(Y_i)$ are independent, take values in $[0,|\rho|^{-1}]$ by~\eqref{eq:qrho-bounds}, and have mean $\mu_\rho$. Define $\bar Y:=\frac1{k-1}\sum_{i=1}^{k-1}q_\rho(Y_i)$, so that Hoeffding's inequality gives
    \[
    \P\left(\left|\bar Y-\mu_\rho\right|>\frac\tau2\right)\le2\exp\left(-\frac{(k-1)\rho^2\tau^2}2\right).
    \]
    We have $\frac{k-1}k\bar Y-\mu_\rho=\frac{k-1}k(\bar Y-\mu_\rho)-\frac{\mu_\rho}k$, so by the triangle inequality and $\frac{k-1}k\le1$,
    \[
    |Q_k-\mu_\rho|\stackrel d=\left|\frac{k-1}{k}\bar Y-\mu_\rho\right|\le|\bar Y-\mu_\rho|+\frac{\mu_\rho}k.
    \]
    For $k\ge2\mu_\rho/\tau$, the second term is at most $\tau/2$, so $|Q_k-\mu_\rho|>\tau$ implies $|\bar Y-\mu_\rho|>\tau/2$. Combined with the Hoeffding bound above and $k-1\ge k/2$ for $k\ge2$,
    \begin{align*}
    \P(|Q_k-\mu_\rho|>\tau)&\le\P\left(|\bar Y-\mu_\rho|>\frac\tau2\right)\\
    &\le2\exp\left(-\frac{(k-1)\rho^2\tau^2}2\right)\le2\exp\left(-\frac{k\rho^2\tau^2}4\right).
    \end{align*}
\end{proof}

We now combine these with \cref{cor:gamma-cramer-md} to obtain the moderate-deviations result for the bias-corrected estimator. The proof additionally uses a diagonal argument that turns \cref{lem:signed-soc} into a bound with vanishing error along $t=n/k$, and local stability of $A_0$ near $n/k$, which follows from the uniform convergence theorem for regularly varying functions.

\begin{theorem}\label{thm:hill-bias-corrected}
    Suppose $U$ satisfies~\eqref{eq:soc} with rate $A(t)$ and $\rho<0$. Let $\theta_k>0$ satisfy
    \[
    \theta_k\to0,\qquad k\theta_k^2\to\infty,\qquad |A(n/k)|=\bigO(\theta_k).
    \]
    Then
    \[
    -\log\P(\alpha\tilde H_{k,n}>1+\theta_k)\sim\frac k2\theta_k^2\qquad\text{and}\qquad -\log\P(\alpha\tilde H_{k,n}<1-\theta_k)\sim\frac k2\theta_k^2,
    \]
    as $n\gg k\gg1$.
\end{theorem}
\begin{proof}
    Set $r_k:=\sqrt{\theta_k}$ and $\mathcal C_k:=\{|E_{k,n}-\log(n/k)|\le r_k\}$. Since $r_k\to0$ and $kr_k^2=k\theta_k=(k\theta_k^2)/\theta_k\to\infty$, \cref{lem:ekn-concentration} (applied with $r=r_k$) gives
    \[
    \P(\mathcal C_k^c)\le2\exp\left(-\frac{kr_k^2}{32}\right)=2\exp\left(-\frac{k\theta_k}{32}\right).
    \]
    Fix $M<\infty$ such that $|A(n/k)|\le M\theta_k$ for all sufficiently large $k$, which exists since $|A(n/k)|=\bigO(\theta_k)$.

    \paragraph{A uniform expansion for \texorpdfstring{$\alpha H_{k,n}$}{alpha H\_\{k,n\}}.} On $\mathcal C_k$, $e^{E_{k,n}}\ge \ell_k:=(n/k)e^{-r_k}\to\infty$. Choose $\varepsilon_k\downarrow0$ slowly enough that $t_0(\varepsilon_k)\le \ell_k$ for all large $k$ (possible since $t_0(\varepsilon)$ is finite for each fixed $\varepsilon\in(0,1]$ and $\ell_k\to\infty$). Then \cref{lem:signed-soc} applies with $t=e^{E_{k,n}}$ and $\varepsilon=\varepsilon_k$ on $\mathcal C_k$, uniformly over $y\ge0$. Since $A_0\in\mathcal{RV}_\rho$ with $\rho<0$ and $\ell_k/(n/k)\to1$, the standard (two-sided) Potter bound gives
    \[
    \sup_{t\ge \ell_k}|A_0(t)|=\bigO(|A_0(\ell_k)|)=\bigO(|A(n/k)|)=\bigO(\theta_k).
    \]
    Hence the error term in \cref{lem:signed-soc} is $\bigO(\varepsilon_k\theta_k+\theta_k^2)=o(\theta_k)$, deterministically on $\mathcal C_k$, since $\varepsilon_k\to0$.

    We average this bound over $E_{i,n}-E_{k,n}$, $i=1,\dots,k-1$. Using the identity $\sum_{i=1}^{k-1}\left(E_{i,n}-E_{k,n}\right)=\Gamma_{k-1}$ from the proof of \cref{lem:hill-bound}, we get, deterministically on $\mathcal C_k$,
    \[
    \alpha H_{k,n}=\frac{\Gamma_{k-1}}k+\alpha A_0(e^{E_{k,n}})Q_k+o(\theta_k).
    \]

    \paragraph{Local stability of \texorpdfstring{$A_0$}{A\_0}.} Write $t_n:=n/k$. The uniform convergence theorem for $A_0\in\mathcal{RV}_\rho$ gives $\sup_{s\in[-1,1]}|A_0(t_ne^s)/A_0(t_n)-e^{\rho s}|\to0$. Since $r_k\to0$ and $A_0\sim A$, this gives
    \[
    \sup_{|s|\le r_k}\left|\frac{A_0(t_ne^s)}{A(t_n)}-1\right|\to0.
    \]
    We apply this to $s=E_{k,n}-\log(n/k)$, which satisfies $|s|\le r_k$ on $\mathcal C_k$. Together with $|A(n/k)|\le M\theta_k$, this gives, deterministically on $\mathcal C_k$ and for all large $k$,
    \[
    |A_0(e^{E_{k,n}})|\le2|A(n/k)|\le2M\theta_k,\qquad |A_0(e^{E_{k,n}})-A(n/k)|=o(\theta_k).
    \]
    By~\eqref{eq:bias-corrected-hill}, $\alpha\tilde H_{k,n}=\alpha H_{k,n}-\alpha A(n/k)/(1-\rho)$, so on $\mathcal C_k$,
    \[
    \alpha\tilde H_{k,n}-1=\frac{\Gamma_{k-1}}k-1+\alpha A_0(e^{E_{k,n}})(Q_k-\mu_\rho)+\alpha\mu_\rho\left(A_0(e^{E_{k,n}})-A(n/k)\right)+o(\theta_k).
    \]
    On $\mathcal C_k\cap\{|Q_k-\mu_\rho|\le r_k\}$, the two middle terms are bounded by $\bigO(\theta_k)\cdot r_k+o(\theta_k)=o(\theta_k)$, so
    \begin{equation}\label{eq:hill-bias-corrected-expansion}
    \alpha\tilde H_{k,n}-1=\frac{\Gamma_{k-1}}k-1+o(\theta_k).
    \end{equation}
    Since $r_k\to0$ and $kr_k^2=k\theta_k\to\infty$, we have $kr_k=(kr_k^2)/r_k\to\infty$, so $k\ge2\mu_\rho/r_k$ for sufficiently large $k$. Hence \cref{lem:correction-concentration} applies with $\tau=r_k$, giving $\P(|Q_k-\mu_\rho|>r_k)\le2\exp(-\rho^2\theta_kk/4)$.

    \paragraph{Upper tail.} On $\mathcal C_k\cap\{|Q_k-\mu_\rho|\le r_k\}$, we can use~\eqref{eq:hill-bias-corrected-expansion} to show $\{\alpha\tilde H_{k,n}-1>\theta_k\}\subseteq\{\Gamma_{k-1}\ge(k-1)(1+\theta_k')\}$, where $\theta_k'$ is defined by
    \[
    1+\theta_k'=\frac k{k-1}\left(1+\theta_k-o(\theta_k)\right),
    \]
    so $\theta_k'\sim\theta_k$. We apply \cref{cor:gamma-cramer-md} and take a union bound,
    \[
    \P(\alpha\tilde H_{k,n}-1>\theta_k)\le\P(\Gamma_{k-1}\ge(k-1)(1+\theta_k'))+\P(\mathcal C_k^c)+\P(|Q_k-\mu_\rho|>r_k).
    \]
    Since $k\theta_k\gg k\theta_k^2$, the last two terms are negligible. \cref{cor:gamma-cramer-md} gives $-\log\P(\Gamma_{k-1}\ge(k-1)(1+\theta_k'))\sim\tfrac k2\theta_k'^2\sim\tfrac k2\theta_k^2$, so
    \[
    \limsup_{n\gg k\gg1}\frac1{k\theta_k^2}\log\P(\alpha\tilde H_{k,n}-1>\theta_k)\le-\tfrac12.
    \]
    For the matching lower bound, the reverse inclusion $\{\Gamma_{k-1}\ge(k-1)(1+\theta_k'')\}\cap\mathcal C_k\cap\{|Q_k-\mu_\rho|\le r_k\}\subseteq\{\alpha\tilde H_{k,n}-1>\theta_k\}$ holds for $\theta_k''\sim\theta_k$ defined symmetrically. Since $\P(A\cap B\cap C)\ge\P(A)-\P(B^c)-\P(C^c)$ and $\P(B^c),\P(C^c)=o(\P(A))$ (again because $k\theta_k\gg k\theta_k^2$), \cref{cor:gamma-cramer-md} gives $\liminf\frac1{k\theta_k^2}\log\P(\alpha\tilde H_{k,n}-1>\theta_k)\ge-\tfrac12$. This completes the upper tail.

    The proof for the lower tail is analogous, using $\{\alpha\tilde H_{k,n}-1<-\theta_k\}$ and the lower-tail parts of \cref{lem:gamma-cramer,cor:gamma-cramer-md} in place of the upper-tail parts.
\end{proof}
If $U$ satisfies~\eqref{eq:soc} with rate function $A(t)$, then the sign of $A(t)$ is constant for sufficiently large $t$. We denote this sign by $\text{sgn}(A(\infty))\in\{-1,+1\}$.
As in \cref{def:rescaled-hill}, \cref{thm:hill-bias-corrected} yields a conservative estimator. It is now valid throughout the moderate-deviation regime approaching the CLT scale from above, $\theta_k\gg k^{-1/2}$, regardless of the size of $A(n/k)$ relative to $\theta_k$, as long as $|A(n/k)|=\bigO(\theta_k)$, rather than only for $\theta_k\gg|A(n/k)|$:

\begin{corollary}\label{cor:bias-corrected-conservative}
    Define $\hat\alpha^{(\theta),\mathrm{bc}}_{k,n}:=(1-\theta)/\tilde H_{k,n}$ for $\theta\in[0,1)$.
    \begin{enumerate}
    \item If $k^{-1/2}\ll\theta_k\ll1$ and $|A(n/k)|=\bigO(\theta_k)$, then $\hat\alpha^{(\theta_k),\mathrm{bc}}_{k,n}$ is conservative with decay rate $\lambda_k=\tfrac k2\theta_k^2$:
    \[
    -\log\P\left(\hat\alpha^{(\theta_k),\mathrm{bc}}_{k,n}>\alpha\right)\sim\frac k2\theta_k^2.
    \]
    \item Let $\theta_k\ge0$ with $\theta_k\to0$. If $\sqrt k|A(n/k)|\to\infty$ and $\liminf_{n\gg k\gg 1} \theta_k/|A(n/k)|>-\frac{\alpha}{1-\rho}\text{sgn}(A(\infty))$, then $\hat\alpha^{(\theta_k)}_{k,n}$ (without bias-correction) is conservative with decay rate
    \[
    -\log\P\left(\hat\alpha^{(\theta_k)}_{k,n}>\alpha\right)\sim\frac k2\left(\theta_k+\frac{\alpha}{1-\rho}A(n/k)\right)^2.
    \]
    \item In particular, if $\sqrt kA(n/k)\to+\infty$, then the Hill estimator $\hat\alpha_{k,n}=1/H_{k,n}$ (without bias-correction or rescaling) has a decay rate of
    \begin{equation}\label{eq:original-hill-decay}
        -\log\P(\hat\alpha_{k,n}>\alpha)\sim \frac{k\alpha^2}{2(1-\rho)^2}A(n/k)^2.
    \end{equation}
    \end{enumerate}
\end{corollary}
\begin{proof}
    We use $\{\hat\alpha^{(\theta_k),\mathrm{bc}}_{k,n}>\alpha\}=\{\alpha\tilde H_{k,n}<1-\theta_k\}$ and apply \cref{thm:hill-bias-corrected}. For the second claim, we apply \cref{thm:hill-bias-corrected} with $\theta_k'=\theta_k+\frac{\alpha}{1-\rho}A(n/k)$. We verify \[\liminf_{n\gg k\gg 1}\theta_k'/|A(n/k)|=\frac\alpha{1-\rho}\text{sgn}(A(\infty))+\liminf_{n\gg k\gg 1}\theta_k/|A(n/k)|>0,\]
    so that $\theta_k'>0$ and $|A(n/k)|=\bigO(\theta_k')$ eventually, while $\theta_k'\to0$ and $k\theta_k'^2\to\infty$ follow from $\theta_k\to0$, $A(n/k)\to0$ and $\sqrt k|A(n/k)|\to\infty$. Hence, \cref{thm:hill-bias-corrected} applies to $\theta_k'$.
    The claim follows from
    \[
    \{\alpha\tilde H_{k,n}<1-\theta_k'\}=\{\alpha H_{k,n}<1-\theta_k\}=\{\hat\alpha_{k,n}^{(\theta_k)}>\alpha\}.
    \]
    Finally, the last claim follows from the previous one with $\theta_k=0$.
\end{proof}

In practice, the second-order bias $A(n/k)$ and the second-order parameter $\rho$ are unknown, so that the bias-corrected Hill estimator $\tilde H_{k,n}$ is not directly computable. \eqref{eq:original-hill-decay} again highlights that the decay rate is large whenever the bias of the estimator is large. 

\subsection{The case \texorpdfstring{$\rho=0$}{rho=0}}
For $\rho=0$, we have $\int_1^{e^y}u^{\rho-1}du= y$ and the SOC bound from \eqref{eq:haan-soc} merely gives
\begin{equation}\label{eq:haan-soc-rewritten}
    \alpha\log\frac{U(te^y)}{U(t)}\ge y+\alpha\log\left(1-(\varepsilon e^{\delta y}+y)|A_0(t)|\right),
\end{equation}
but $\varepsilon e^{y\delta}$ is not bounded in $y$ for $\delta>0$. This shows that for $\rho=0$, the SOC does not yield a faster Potter rate in terms of $A(t)$. 
Because of this, we need a stronger condition $\theta_k\gg \sup_{t\ge n/k}|A(t)|^{1/(1+\kappa)}$ for some $\kappa>0$ to get similar moderate deviations as in the case $\rho<0$.

\begin{theorem}\label{thm:hill-rho0}
    If $U$ satisfies~\eqref{eq:soc} with rate $A(t)$ and $\rho=0$, then for any sequence $\theta_k\to0$ with $\theta_k\gg k^{-1/2}+\sup_{t\ge n/k}|A(t)|^{1/(1+\kappa)}$ for some $\kappa\in(0,1)$, we have
    \[
        -\log\P(\alpha H_{k,n}>1+\theta_k)\sim \frac{k}{2}\theta_k^2\qquad\text{and}\qquad -\log\P(\alpha H_{k,n}<1-\theta_k)\sim \frac{k}{2}\theta_k^2,
    \]
    as $n\gg k\gg1$.
\end{theorem}
\begin{proof}
    We define $A^*(t)=\sup_{t'\ge t}|A_0(t')|$ so that $A^*(n/k)=\bigO(\theta_k^{1+\kappa})$, where $A_0(t)$ is as given in~\eqref{eq:haan-soc} and the asymptotics follow from the assumption on $\theta_k$ together with $A_0(t)\sim A(t)$. Note that $A^*$ is non-increasing. Moreover, since $A_0\in\mathcal{RV}_0$, the uniform convergence theorem for slowly varying functions gives $\sup_{t'\in[ct,t]}|A_0(t')|\le(1+o(1))|A_0(t)|$ for every fixed $c\in(0,1)$, and hence
    \begin{equation}\label{eq:astar-local}
    A^*(ct)\le(1+o(1))A^*(t),\qquad t\to\infty,
    \end{equation}
    for every fixed $c\in(0,1)$. In particular, $A^*(e^{-\kappa}n/k)=\bigO(\theta_k^{1+\kappa})$ and $A^*(e^{-1}n/k)=\bigO(\theta_k^{1+\kappa})$.
    \paragraph{Lower tail, upper bound.} We first prove the bound
    \[
    \log\mathbb P(\alpha H_{k,n}\le 1-\theta_k)\le -\frac{k\theta_k^2}2\cdot(1+o(1)),
    \]
    directly via a Chernoff bound. Write $X_{i,n}=U(e^{E_{i,n}})$ and, for $s>0$, $g(z,u):=\left(U(e^{u+z})/U(e^u)\right)^{-\alpha s}$, so that $e^{-sk\alpha H_{k,n}}=\prod_{i=1}^k g(E_{i,n}-E_{k,n},E_{k,n})$ (the $i=k$ term contributes $g(0,E_{k,n})=1$). Notice that the product only depends on the (multi-)set of values $\{E_{i,n}-E_{k,n}\}_{i=1}^{k-1}$, which is independent of $E_{k,n}$ and equal in distribution to a set of i.i.d. copies of $E\sim\text{Exp}(1)$. Hence, writing
    \[
    h(u):=\E\left[g(E,u)\right]=\E\left[\left(\frac{U(e^{u+E})}{U(e^{u})}\right)^{-\alpha s}\right],
    \]
    we have $\E\left[e^{-sk\alpha H_{k,n}}\ \middle|\ E_{k,n}\right]=h(E_{k,n})^{k-1}$.
    Combined with the Chernoff bound $\P(\alpha H_{k,n}\le1-\theta_k)=\P(e^{-sk\alpha H_{k,n}}\ge e^{-sk(1-\theta_k)})\le\E[e^{-sk\alpha H_{k,n}}]e^{sk(1-\theta_k)}$, this yields
    \begin{align}\label{eq:hill-md-upper}
\mathbb P(\alpha H_{k,n}\le 1-\theta_k)&\le\E\left[h(E_{k,n})^{k-1}\right]e^{ks(1-\theta_k)}.
\end{align}
We introduce $e_{k,n}$ as the maximum of $\log t_0(\kappa,\kappa/5)$ and $\log(n/k)-\kappa$, where $t_0(\varepsilon,\delta)$ is as given in the proof of~\cref{lem:potter-soc}. This guarantees $e^{u}\ge t_0(\kappa,\kappa/5)$ for $u\ge e_{k,n}$, which allows us to use~\eqref{eq:haan-soc-rewritten} for such $u$. Since $U$ is non-decreasing, we have $g\le1$ and hence $h\le1$, so that
\begin{align*}
\E\left[h(E_{k,n})^{k-1}\right]
\le \P(E_{k,n}<e_{k,n})+\sup_{u\ge e_{k,n}}h(u)^{k-1}.
\end{align*}
The large deviations of $E_{k,n}$ (see \cref{lem:intermediate-ld}) imply $\P(E_{k,n}< e_{k,n})\le e^{-\Theta(k)}$, which will be shown to be negligible.
We will prove that for $s=\theta_k/(1-\theta_k)$, it holds that
\begin{equation}\label{eq:chernoff-soc}
\sup_{u\ge e_{k,n}}h(u)\le\exp\left(-\theta_k-\frac12\theta_k^2+o(\theta_k^2)\right).
\end{equation}
We will make use of~\eqref{eq:haan-soc-rewritten}. We introduce $y_0=y_0(\kappa)$ as the solution of
\begin{equation}\label{eq:y0-def}
A^*(e^{-\kappa}n/k)(\kappa e^{\kappa y/5}+y)=\theta_k^{1+\kappa/2}.
\end{equation}
Since the left-hand side is continuous and increasing in $y$, and $A^*(e^{-\kappa}n/k)=\bigO(\theta_k^{1+\kappa})\ll\theta_k^{1+\kappa/2}$, the solution $y_0$ exists and satisfies $y_0\to\infty$.
Fix $u\ge e_{k,n}$. We bound
\begin{align*}
h(u)
&\le\P(E>y_0)+\E\left[\mathbb1\{E\le y_0\}\left(\frac{U(e^{u+E})}{U(e^{u})}\right)^{-\alpha s}\right].
\end{align*}
Since $A^*$ is non-increasing and $u\ge e_{k,n}\ge\log(n/k)-\kappa$, we have $A^*(e^u)\le A^*(e^{-\kappa}n/k)$. We use the bounds~\eqref{eq:haan-soc-rewritten} and~\eqref{eq:y0-def}, together with the fact that $(1-y)^{-\alpha s}$ is an increasing function for $y\in(0,1)$, to obtain
\begin{align*}
    &\E\left[\mathbb1\{E\le y_0\}\left(\frac{U(e^{u+E})}{U(e^{u})}\right)^{-\alpha s}\right]\\
    &\le \E\left[\mathbb1\{E\le  y_0\}e^{-sE}\left(1-(\kappa e^{\kappa E/5}+E)A^*(e^{u})\right)^{-\alpha s}\right]\\
    &\le \E\left[\mathbb1\{E\le  y_0\}e^{-sE}\left(1-\theta_k^{1+\kappa/2}\right)^{-\alpha s}\right]\\
    &\le \E[e^{-sE}]\cdot \left(1-\theta_k^{1+\kappa/2}\right)^{-\alpha s}\\
    &= (1+s)^{-1}\cdot \exp\left(\bigO\left(s\cdot \theta_k^{1+\kappa/2}\right)\right),
\end{align*}
uniformly in $u\ge e_{k,n}$.
To find the asymptotics of $\P(E>y_0)$, we write
\(
\P(E>y_0)=e^{-y_0},
\)
and use~\eqref{eq:y0-def} with $A^*(e^{-\kappa}n/k)=\bigO(\theta_k^{1+\kappa})$ to obtain
\begin{align*}
e^{-y_0}
&\sim \left(\frac{\kappa A^*(e^{-\kappa}n/k)}{\theta_k^{1+\kappa/2}}\right)^{5/\kappa}=\bigO\left(\theta_k^{5/2}\right).
\end{align*}
Putting everything together, we have obtained the bound
\begin{align*}
\sup_{u\ge e_{k,n}}h(u)
\le &\bigO\left(\theta_k^{5/2}\right)+(1+s)^{-1}\cdot \exp\left(\bigO\left(s\cdot \theta_k^{1+\kappa/2}\right)\right).
\end{align*}
We substitute $s=\theta_k/(1-\theta_k)\sim\theta_k$, so that
\[
-\log(1+s)=\log(1-\theta_k)=-\theta_k-\frac12\theta_k^2+\bigO(\theta_k^3),
\]
which leads to
\[
\log \sup_{u\ge e_{k,n}}h(u)\le -\theta_k-\frac12\theta_k^2+\bigO(\theta_k^{2+\kappa/2}),
\]
which proves~\eqref{eq:chernoff-soc}.

We now move on to the remainder of the proof and substitute~\eqref{eq:chernoff-soc} into~\eqref{eq:hill-md-upper} to obtain
\begin{align*}
    \mathbb P(\alpha H_{k,n}\le 1-\theta_k)
    &\le\left(e^{-\Theta(k)}+\exp\left(-(k-1)\left(\theta_k+\frac12\theta_k^2+o(\theta_k^2)\right)\right)\right)e^{k\theta_k}\\
    &=\exp\left(-\frac k2\theta_k^2+o(k\theta_k^2)\right),
\end{align*}
so that, indeed
\[
\log\mathbb P(\alpha H_{k,n}\le 1-\theta_k)\le -\frac{k\theta_k^2}2\cdot(1+o(1)),
\]
which proves the upper bound.

\emph{Lower bound.} We construct a change of measure on the top $k-1$ exponential spacings $\Delta_1,\dots,\Delta_{k-1}$, leaving $\Delta_k,\dots,\Delta_n$, and hence $E_{k,n}=\sum_{i=k}^n\Delta_i/i$, unchanged. Recall from~\eqref{eq:exp-spacing-coupling} that
\[
\frac{X_{i,n}}{X_{k,n}}=\frac{U(e^{E_{k,n}+Y_i})}{U(e^{E_{k,n}})},\qquad Y_i:=\sum_{j=i}^{k-1}\frac{\Delta_j}j,\quad i\in[k-1],
\]
so that $H_{k,n}$ is a function of $(\Delta_{1:k-1},E_{k,n})$, where $\Delta_{1:k-1}$ is independent of $E_{k,n}$.
Fix $\varepsilon>0$ and write $m_k:=1-(1+\varepsilon)\theta_k$ and $c_k:=\frac{(1+\varepsilon)\theta_k}{1-(1+\varepsilon)\theta_k}$, so that $1/m_k-c_k=1$. Let $\Delta_1',\dots,\Delta_{k-1}'$ be i.i.d.\ exponential random variables with mean $m_k$, independent of $\Delta_k,\dots,\Delta_n$. Tilting the density of $\Delta_i'$ by $e^{c_k\Delta_i'}$ yields a standard exponential density, since $m_k^{-1}e^{-x/m_k}e^{c_kx}=m_k^{-1}e^{-x}$, and the corresponding normalizing constant is $\E[e^{c_k\Delta_i'}]=1/m_k$. We define $H'_{k,n}$ as the Hill estimator applied to the order statistics
\[
    X_{i,n}'=U\left(\exp\left(E_{k,n}+\sum_{j=i}^{k-1}\frac{\Delta_j'}{j}\right)\right),\ i\in[k-1],\qquad X'_{k,n}=X_{k,n}=U\left(e^{E_{k,n}}\right).
\]
Since the tilted law of $(\Delta'_{1:k-1},E_{k,n})$ coincides with the law of $(\Delta_{1:k-1},E_{k,n})$, we obtain
\begin{align*}
    \mathbb P(\alpha H_{k,n}\le1-\theta_k)
    &=\frac{\mathbb E\left[\mathbb1\{\alpha H'_{k,n}\le1-\theta_k\}\exp\left(c_k\sum_{i\in[k-1]}\Delta'_i\right)\right]}{\mathbb E\left[\exp\left(c_k\sum_{i\in[k-1]}\Delta'_i\right)\right]}\\
    &=\left(1-(1+\varepsilon)\theta_k\right)^{k-1}\\
    &\quad\times\mathbb E\left[\mathbb1\{\alpha H'_{k,n}\le1-\theta_k\}\exp\left(c_k\sum_{i\in[k-1]}\Delta'_i\right)\right].
\end{align*}
We introduce the event
\[
\mathcal L_k=\left\{\frac1{k-1}\sum_{i\in[k-1]}\Delta_i'\ge 1-(1+2\varepsilon)\theta_k\right\}.
\]
It holds that $\P(\mathcal L_k)\to1,$ since
\[
\frac1{k-1}\sum_{i\in[k-1]}\Delta_i'=1-(1+\varepsilon)\theta_k+\bigO_\P(k^{-1/2})
\]
by the CLT, while $\varepsilon\theta_k\gg k^{-1/2}$.
On $\mathcal L_k$, we have
\[
c_k\sum_{i\in[k-1]}\Delta'_i\ge(k-1)(1+\varepsilon)\theta_k\frac{1-(1+2\varepsilon)\theta_k}{1-(1+\varepsilon)\theta_k}=(k-1)(1+\varepsilon)\theta_k\left(1-\varepsilon\theta_k+o(\theta_k)\right),
\]
so that
\begin{align}
    &\P(\alpha H_{k,n}\le 1-\theta_k)\nonumber\\
    \ge &\P(\alpha H'_{k,n}\le1-\theta_k,\mathcal L_k)\nonumber\\
    &\times\exp\left((k-1)\left(\log\left(1-(1+\varepsilon)\theta_k\right)+(1+\varepsilon)\theta_k-\varepsilon(1+\varepsilon)\theta_k^2+o(\theta_k^2)\right)\right)\nonumber\\
    =& \P(\alpha H'_{k,n}\le1-\theta_k,\mathcal L_k)\exp\left(-k\left(\frac12(1+\varepsilon)^2\theta_k^2+\varepsilon (1+\varepsilon)\theta_k^2+o(\theta_k^2)\right)\right), \label{eq:rho0lower-tail-lower-exponent}
\end{align}
where we used $(k-1)\theta_k^2=k\theta_k^2+o(k\theta_k^2)$.
What remains to show is that $\P(\alpha H_{k,n}'\le 1-\theta_k)\to1.$
Similarly to the derivation of~\eqref{eq:haan-soc-rewritten}, we rewrite~\eqref{eq:haan-soc} with $\varepsilon=\kappa$ and $\delta=\kappa/5$ (using that $\rho=0$, so that $\int_1^{e^y}u^{\rho-1}du=y$) to
\begin{equation}\label{eq:haan-soc-rewritten-upper}
    \alpha\log\frac{U(te^y)}{U(t)}-y\le\alpha\log\left(1+yA^*(t)+\kappa A^*(t)e^{\kappa y/5}\right)\le \alpha yA^*(t)+\alpha\kappa A^*(t)e^{\kappa y/5},
\end{equation}
valid for all $y\ge0$ and $t\ge t_0(\kappa,\kappa/5)$.
Write $Y_i':=\sum_{j=i}^{k-1}\Delta_j'/j$ for $i\in[k-1]$. On $\{E_{k,n}\ge\log(n/k)-1\}$, which has probability tending to one by \cref{lem:intermediate-ld}, we have $e^{E_{k,n}}\ge t_0(\kappa,\kappa/5)$ for all sufficiently large $k$, and~\eqref{eq:haan-soc-rewritten-upper} with $t=e^{E_{k,n}}$ and $y=Y_i'$ leads to
\begin{align*}
    \alpha H_{k,n}'\le \frac{1+\alpha A^*(e^{E_{k,n}})}k\sum_{i\in[k-1]}Y'_{i}+\frac{\alpha\kappa A^*(e^{E_{k,n}})}{k}\sum_{i\in[k-1]}e^{\kappa Y'_i/5}.
\end{align*}
By the Rényi representation, the (multi-)set $\{Y_i'\}_{i\in[k-1]}$ is equal in distribution to $k-1$ i.i.d.\ exponential random variables with mean $m_k=1-(1+\varepsilon)\theta_k$. Hence, by Chebyshev's inequality, $\frac1k\sum_{i\in[k-1]}Y_i'=1-(1+\varepsilon)\theta_k+\bigO_\P(k^{-1/2})$ (the factor $(k-1)/k$ contributes $\bigO_\P(1/k)=o_\P(k^{-1/2})$), while $\frac1k\sum_{i\in[k-1]}e^{\kappa Y_i'/5}=\bigO_\P(1)$ by Markov's inequality, since $\E[e^{\kappa Y_i'/5}]=(1-\kappa m_k/5)^{-1}$ is bounded. This gives
\begin{align*}
    \alpha H_{k,n}'
    \le&\left(1+\alpha A^*(e^{E_{k,n}})\right)\left(1-(1+\varepsilon)\theta_k+\bigO_\P(k^{-1/2})\right)+\bigO_\P\left(A^*(e^{E_{k,n}})\right)\\
    =&1-(1+\varepsilon)\theta_k+\bigO_\P\left(k^{-1/2}+A^*(e^{E_{k,n}})\right).
\end{align*}
Since $A^*$ is non-increasing, on $\{E_{k,n}\ge\log(n/k)-1\}$ we have $A^*(e^{E_{k,n}})\le A^*(e^{-1}n/k)=\bigO(\theta_k^{1+\kappa})$ by~\eqref{eq:astar-local}, so that $A^*(e^{E_{k,n}})=\bigO_\P(\theta_k^{1+\kappa})=o_\P(\theta_k)$. Together with $\theta_k\gg k^{-1/2}$, the above is $1-\theta_k-(\varepsilon+o_\P(1))\theta_k$, so that $\P(\alpha H_{k,n}'\le 1-\theta_k)\to1.$
We conclude that
\[
\P(\alpha H_{k,n}\le 1-\theta_k)\ge \exp\left(-k\left(\frac12(1+\varepsilon)^2\theta_k^2+\varepsilon (1+\varepsilon)\theta_k^2+o(\theta_k^2)\right)+o(1)\right).
\]
Sending $\varepsilon\downarrow0$ results in the desired
\[
\P(\alpha H_{k,n}\le1-\theta_k)\ge \exp(-\frac12k\theta_k^2(1+o(1))),
\]
which, combined with the upper bound above, proves $\lim_{n\gg k\gg1}\frac{-2}{k\theta_k^2}\log\P(\alpha H_{k,n}<1-\theta_k)=1$ for $\rho=0$.

\paragraph{Upper tail.} We now turn to $\P(\alpha H_{k,n}>1+\theta_k)$. Since $\lambda_k:=\frac k2\theta_k^2=o(k)$, \cref{lem:intermediate-ld} allows us to pick $\eta=e^\kappa$ so that $\P(E_{k,n}<e_{k,n})\le e^{-\Theta(k)}$ is negligible and \cref{cor:conditioning} lets us reduce to bounding $\P(\alpha H_{k,n}>1+\theta_k\mid E_{k,n}\ge e_{k,n})$ from above and below.

\emph{Upper bound.} For $s>0$, write $g(z,u):=\left(U(e^{u+z})/U(e^u)\right)^{\alpha s}$ and $h(u):=\E[g(E,u)]$. Define $\E'[\cdot]:=\E[\cdot\mid E_{k,n}\ge e_{k,n}]$, so that $\E'\left[e^{sk\alpha H_{k,n}}\right]=\E'\left[h(E_{k,n})^{k-1}\right]$, so that
\[
\P(\alpha H_{k,n}>1+\theta_k\mid E_{k,n}\ge e_{k,n})\le\E'\left[h(E_{k,n})^{k-1}\right]e^{-sk(1+\theta_k)}.
\]
Since $e^{E_{k,n}}\ge t_0(\kappa,\kappa/5)$ on $\{E_{k,n}\ge e_{k,n}\}$, we now use~\eqref{eq:haan-soc-rewritten-upper} in place of~\eqref{eq:haan-soc-rewritten}. Let $y_0':=y_0$ be as in~\eqref{eq:y0-def}, so that $e^{-y_0'}=\bigO(\theta_k^{5/2})$, and split
\[
h(u)=\E\left[g(E,u)\1\{E\le y_0'\}\right]+\E\left[g(E,u)\1\{E> y_0'\}\right].
\]
For the first term, for $u\ge e_{k,n}$ and $0\le y\le y_0'$, \eqref{eq:haan-soc-rewritten-upper} together with $A^*(e^u)\le A^*(e^{-\kappa}n/k)$ and~\eqref{eq:y0-def} gives
\[
\alpha\log\frac{U(e^{u+y})}{U(e^u)}\le y+\alpha\left(\kappa e^{\kappa y_0'/5}+y_0'\right)A^*(e^{-\kappa}n/k)=y+\alpha\theta_k^{1+\kappa/2},
\]
so that, using $\E[e^{sE}]=(1-s)^{-1}$ for $s<1$,
\[
\E\left[g(E,u)\1\{E\le y_0'\}\right]\le e^{\alpha s\theta_k^{1+\kappa/2}}\E\left[e^{sE}\right]=(1-s)^{-1}\exp\left(\bigO\left(s\theta_k^{1+\kappa/2}\right)\right).
\]
Unlike in the lower tail, the second term cannot be bounded by $\P(E>y_0')$, since $g(z,u)\ge1$ now contains a positive power. Instead, we use the Potter bound from \cref{lem:potter-rewritten}: since $e^{e_{k,n}}\to\infty$, we have $\varepsilon(e^{e_{k,n}})\le1$ for all sufficiently large $k$, so that $U(e^{u+y})/U(e^u)\le2e^{2y/\alpha}$ for all $u\ge e_{k,n}$ and $y\ge0$. Hence, for $s\le1/4$,
\begin{align*}
\E\left[g(E,u)\1\{E>y_0'\}\right]&\le2^{\alpha s}\int_{y_0'}^\infty e^{2sy}e^{-y}\,dy\\
&=\frac{2^{\alpha s}}{1-2s}e^{-(1-2s)y_0'}\le\frac{2^{\alpha s}}{1-2s}\left(e^{-y_0'}\right)^{1-2s}.
\end{align*}
We will take $s=\theta_k/(1+\theta_k)\to0$ below. Since $e^{-y_0'}=\bigO(\theta_k^{5/2})$ and $\theta_k^{-5s}=e^{-5s\log\theta_k}\to1$, the right-hand side is $\bigO(\theta_k^{5/2-5s})=\bigO(\theta_k^{5/2})$. Combining both terms and using $(1-s)^{-1}\ge1$ and $\kappa<1$, we obtain, uniformly over $u\ge e_{k,n}$,
\[
h(u)\le(1-s)^{-1}\left(\exp\left(\bigO(\theta_k^{2+\kappa/2})\right)+\bigO\left(\theta_k^{5/2}\right)\right)=(1-s)^{-1}\exp\left(\bigO(\theta_k^{2+\kappa/2})\right),
\]
so that $\E'[h(E_{k,n})^{k-1}]\le\left((1-s)^{-1}\exp(\bigO(\theta_k^{2+\kappa/2}))\right)^{k-1}$. Substitute $s=\theta_k/(1+\theta_k)$, so that $1-s=(1+\theta_k)^{-1}$ and $sk(1+\theta_k)=k\theta_k$. Since
\[
-\log(1-s)=\log(1+\theta_k)=\theta_k-\frac12\theta_k^2+\bigO(\theta_k^3),
\]
combining gives
\begin{align*}
\P(\alpha H_{k,n}>1+\theta_k\mid E_{k,n}\ge e_{k,n})
&\le\exp\left((k-1)\left(\theta_k-\tfrac12\theta_k^2+o(\theta_k^2)\right)-k\theta_k\right)\\
&=\exp\left(-\tfrac k2\theta_k^2+o(k\theta_k^2)\right),
\end{align*}
using $\theta_k=o(k\theta_k^2)$ as before, and \cref{cor:conditioning} gives $\log\P(\alpha H_{k,n}>1+\theta_k)\le-\frac{k\theta_k^2}2(1+o(1))$.

\emph{Lower bound.} We mirror the tilting argument of the lower tail with $\theta_k$ replaced by $-\theta_k$. Let $\Delta_1',\dots,\Delta_{k-1}'$ be i.i.d.\ exponential random variables with mean $\beta_k:=1+(1+\varepsilon)\theta_k$, independent of $\Delta_k,\dots,\Delta_n$, so that tilting by $e^{-\frac{(1+\varepsilon)\theta_k}{1+(1+\varepsilon)\theta_k}\Delta_i'}$ yields standard exponentials, and define $H_{k,n}'$ from $(\Delta_{1:k-1}',E_{k,n})$ as before.
Defining $\mathcal L_k^+:=\left\{\frac1{k-1}\sum_{i\in[k-1]}\Delta_i'\le1+(1+2\varepsilon)\theta_k\right\}$, we can again conclude $\P(\mathcal L_k^+)\to1$ by the same CLT argument, since $\varepsilon\theta_k\gg k^{-1/2}$.
By repeating the steps from the lower bound of the lower tail, we obtain an analogous expression as~\eqref{eq:rho0lower-tail-lower-exponent}:
\begin{align*}
\P(\alpha H_{k,n}>1+\theta_k)
\ge&\P(\alpha H_{k,n}'>1+\theta_k,\mathcal L_k^+)\\
&\times\exp\left(-k\left(\tfrac12(1+\varepsilon)^2\theta_k^2+\varepsilon(1+\varepsilon)\theta_k^2+o(\theta_k^2)\right)\right).
\end{align*}
It remains to show $\P(\alpha H_{k,n}'>1+\theta_k)\to1$. Write $y_i:=\sum_{j=i}^{k-1}\Delta_j'/j$ for $i\in[k-1]$; by the Rényi representation, the (multi-)set $\{y_i\}_{i\in[k-1]}$ is equal in distribution to $k-1$ i.i.d.\ exponential random variables with mean $\beta_k>1$. We use~\eqref{eq:haan-soc-rewritten} with $\varepsilon=\kappa$ and $\delta=\kappa/5$ to bound $\alpha H_{k,n}'$ from below. This bound is valid only when $(\kappa e^{\kappa y/5}+y)A^*(t)<1$, so it cannot simply be summed over all $k-1$ gaps $y_i$: their maximum is of order $\log k$, which can exceed the validity threshold $\Theta(\log(1/A^*(\cdot)))$ if $A^*$ decays slowly enough (as is permitted when $\rho=0$). We distinguish at $y_i\le y_0'$ and $y_i>y_0'$. For $y\le y_0'$ and under $\{E_{k,n}\ge e_{k,n}\}$, we have $A^*(e^{E_{k,n}})\le A^*(e^{-\kappa}n/k)$, so that
\[
(\kappa e^{\kappa y/5}+y)A^*(e^{E_{k,n}})\le(\kappa e^{\kappa y_0'/5}+y_0')A^*(e^{-\kappa}n/k)=\theta_k^{1+\kappa/2}\to0,
\]
so~\eqref{eq:haan-soc-rewritten} applies there. Using $\log(1-x)\ge-2x$ for $x\in[0,\tfrac12]$, we conclude that for $0\le y\le y_0'$,
\begin{align*}
\alpha\log\frac{U(e^{E_{k,n}+y})}{U(e^{E_{k,n}})}
&\ge y-2\alpha\left(\kappa e^{\kappa y/5}+y\right)A^*(e^{E_{k,n}})\\
&\ge y-2\alpha \theta_k^{1+\kappa/2}.
\end{align*}
For $y_i>y_0'$ we instead use the trivial bound $\alpha\log\bigl(U(e^{E_{k,n}+y_i})/U(e^{E_{k,n}})\bigr)\ge0$, valid unconditionally since $U$ is non-decreasing. Combining,
\begin{align*}
\alpha H_{k,n}'
&\ge\frac1k\sum_{i=1}^{k-1}y_i\mathbb1\{y_i\le y_0'\}-\frac{2\alpha \theta_k^{1+\kappa/2}}k\sum_{i=1}^{k-1}\mathbb1\{y_i\le y_0'\}\\
&\ge \frac1k\sum_{i=1}^{k-1}y_i\mathbb1\{y_i\le y_0'\}-2\alpha\theta_k^{1+\kappa/2}.
\end{align*}
Notice that the sum is an average of $k-1$ i.i.d.\ truncated exponential variables. By the exact identity
\[
\E\left[Y\mathbb1\{Y\le y_0'\}\right]=\beta_k-(\beta_k+y_0')e^{-y_0'/\beta_k},\qquad Y\sim\text{Exp}(1/\beta_k),
\]
these truncated random variables have expectation $\beta_k-(\beta_k+y_0')e^{-y_0'/\beta_k}$ and bounded variance, so that the deviations of their average are $\bigO_\P(k^{-1/2})=o_\P(\theta_k)$ by Chebyshev's inequality (the factor $(k-1)/k$ again contributes $\bigO_\P(1/k)=o_\P(k^{-1/2})$). For sufficiently large $k$, we have $\beta_k+y_0'\le e^{y_0'/4}$ and $1/\beta_k\ge3/4$, since $y_0'\to\infty$ and $\beta_k\to1$, so that $(\beta_k+y_0')e^{-y_0'/\beta_k}\le e^{-y_0'/2}=\bigO(\theta_k^{5/4})=o(\theta_k)$. Together, this results in
\[
\alpha H_{k,n}'\ge \beta_k+o_\P(\theta_k)=1+(1+\varepsilon+o_\P(1))\theta_k,
\]
so that $\P(\alpha H_{k,n}'>1+\theta_k)\to1,$ as required.
We conclude that
\[
\P(\alpha H_{k,n}>1+\theta_k)\ge\exp\left(-k\left(\tfrac12(1+\varepsilon)^2\theta_k^2+\varepsilon(1+\varepsilon)\theta_k^2+o(\theta_k^2)\right)+o(1)\right),
\]
and sending $\varepsilon\downarrow0$ gives $\P(\alpha H_{k,n}>1+\theta_k)\ge\exp(-\frac12k\theta_k^2(1+o(1)))$, which, combined with the upper bound above, proves $-\log\P(\alpha H_{k,n}>1+\theta_k)\sim \frac{k}{2}\theta_k^2$ for $\rho=0$, which completes the proof.
\end{proof}

%% file: appendix.tex
\section{Large deviations of Pickands estimator}\label{apx:pickands}
Let $k$ be divisible by $4$. Pickands estimator is given\footnote{Typically, the Pickands estimator is given as a function of $X_{k,n},X_{2k,n},X_{4k,n}$, but we divide by $4$ so that the estimator uses the same largest $k$ observations as the Hill estimator.} by
\begin{equation}\label{eq:pickands}
    P_{k,n}=\frac{\log\left(\frac{X_{k/4,n}}{X_{k/2,n}}-1\right)-\log\left(1-\frac{X_{k,n}}{X_{k/2,n}}\right)}{\log 2}.
\end{equation}
Pickands estimator is consistent (i.e., $P_{k,n}\stackrel\P\to\frac1\alpha$) since both $\frac{X_{k/4,n}}{X_{k/2,n}}$ and $\frac{X_{k/2,n}}{X_{k,n}}$ converge in probability to $2^{1/\alpha}$.

\begin{theorem}\label{thm:pickands}
    Consider the estimator $\hat\alpha_{k,n}^{(P)}(\theta)=\frac{1-\theta}{P_{k,n}}$. The decay rate of $\hat\alpha_{k,n}^{(P)}(\theta)$ is at most
    \[
    \lambda_k=\frac k4\left(-\theta\log2-\log(2-2^{\theta})\right).
    \]
    For $\theta\in(0,1),$ this is strictly less than the decay rate of the Hill estimator that is proven in~\cref{cor:rescaled-hill-ld}.
\end{theorem}
\begin{proof}
    Recall that
    \[
    X_{k,n}=U(e^{E_{k,n}}),
    \]
    where $E_{k,n}$ is the $k$-th largest among $n$ independent standard exponentials.
    Then
    \[
    \frac{X_{k/4,n}}{X_{k/2,n}}=\frac{U(e^{M_k}e^{E_{k/2,n}})}{U(e^{E_{k/2,n}})},
    \]
    where $M_k=E_{k/4,n}-E_{k/2,n}\stackrel d=E_{k/4,k/2-1}$ corresponds to the median of $k/2-1$ independent exponentials.
    
    We have $M_k\stackrel \P\to\log 2$. Additionally, by the Rényi representation, $M_k=\sum_{i=k/4}^{k/2-1}\Delta_i/i$ is a function of $\Delta_{k/4},\dots,\Delta_{k/2-1}$ only, so that $M_k$ is independent of the pair $(E_{k/2,n}-E_{k,n},E_{k,n})$, which is a function of $\Delta_{k/2},\dots,\Delta_n$. Moreover, $E_{k/2,n}\stackrel\P\to\infty$. Since $M_k\ge0$ and, by \cref{lem:potter-rewritten}, $U(e^{m}e^{t})/U(e^t)=e^{m/\alpha}(1+o(1))$ uniformly in $m\in[0,\log2]$ as $t\to\infty$, it follows that
    \[
    \frac{X_{k/4,n}}{X_{k/2,n}}\stackrel\P\sim e^{\frac1\alpha M_k},
    \]
    and, by the independence just noted, this holds also conditionally on any event of the form $\{M_k\le m\}$ with $m\le\log2$, as does $X_{k/2,n}/X_{k,n}\stackrel\P\to2^{1/\alpha}$.
    We are interested in the large deviations event $\{M_k<(1-\theta)\log2\}$. For each exponential, the probability that it is bigger than $(1-\theta)\log2$ is $(1/2)^{1-\theta}$. Since $M_k$ is the $(k/4)$-th largest of $k/2-1$ exponentials, $M_k$ is below $(1-\theta)\log2$ if and only if at most $k/4-1$ of them exceed this value. This leads to
    \begin{align*}
    \P(M_k<(1-\theta)\log2)
    &=\P(\text{Bin}(k/2-1,(1/2)^{1-\theta})\le k/4-1)\\
    &\ge \P(\text{Bin}(k/2-1,(1/2)^{1-\theta})=k/4-1)\\
    &\ge \frac{e^{-1/6}}{\sqrt{2\pi(k/4-1)}}\exp\left(-\left(\frac k2-1\right)D\left(a_k\middle\|\left(\frac12\right)^{1-\theta}\right)\right),
    \end{align*}
    where $a_k:=\frac{k/4-1}{k/2-1}$ and the last step is an application of~\cref{lem:bin-lower}.
    Since $a_k\to\frac12$ and $a\mapsto D(a\|p)$ is continuous on $(0,1)$ for fixed $p\in(0,1)$, we have
    \[
    \left(\frac k2-1\right)D\left(a_k\middle\|\left(\frac12\right)^{1-\theta}\right)= \frac k2D\left(\frac12\middle\|\left(\frac12\right)^{1-\theta}\right)+o(k).
    \]
    We now write
    \begin{align*}
        \P(\alpha P_{k,n}<1-\theta)\ge\P(M_k<(1-\theta)\log2)\P(\alpha P_{k,n}<1-\theta\ |\ M_k<(1-\theta)\log2).
    \end{align*}
    Now, using the fact that $\alpha P_{k,n}$ is increasing in $X_{k/4,n}/X_{k/2,n}$ and the fact that, conditionally on $\{M_k<(1-\theta)\log2\}$, $X_{k/4,n}/X_{k/2,n}\stackrel\P\sim e^{\frac1\alpha M_k}\le 2^{(1-\theta)/\alpha}$ while $X_{k/2,n}/X_{k,n}\stackrel\P\to2^{1/\alpha}$ (as shown above), we have
    \begin{align*}
    \alpha P_{k,n}
    &\le \alpha\frac{\log\left(2^{(1-\theta)/\alpha}-1\right)-\log\left(1-2^{-1/\alpha}\right)}{\log 2}+o_\P(1)\\
    &=1-\theta+\alpha\frac{\log(1-2^{-\frac{1-\theta}\alpha})-\log(1-2^{-\frac{1}\alpha})}{\log 2}+o_\P(1).
    \end{align*}
    Now, since $\log(1-2^{-\frac{1-\theta}\alpha})<\log(1-2^{-\frac{1}\alpha})$, the above is smaller than $1-\theta$ with high probability, so that
    \[
    \P(\alpha P_{k,n}<1-\theta\ |\ M_k<(1-\theta)\log2)\to1.
    \]
    We conclude that 
    \[
    \P(\alpha P_{k,n}<1-\theta)\ge \exp\left(-\frac k2D\left(\frac12\middle\|\left(\frac12\right)^{1-\theta}\right)+o(k)\right),
    \]
    while
    \[
D\left(\frac12\middle\|\left(\frac12\right)^{1-\theta}\right)=-\frac\theta2\log2-\frac12\log(2-2^\theta),
    \]
    which completes the proof. The fact that this is strictly worse than the Hill estimator is proved in~\cref{lem:pickands-rate}.
\end{proof}

\begin{lemma}\label{lem:pickands-rate}
    The Pickands rate is strictly worse than that of the Hill estimator, i.e.,
    \[
    -\frac\theta4\log2-\frac14\log(2-2^\theta)<-\theta-\log(1-\theta),
    \]
    for all $\theta\in(0,1).$
\end{lemma}
\begin{proof}
    We write
    \[
    h(\theta)=\theta\log2+\log(2-2^\theta)-4\theta-4\log(1-\theta),
    \]
    and will show that $h(\theta)>0.$ Note that $h(0)=0$. We calculate
    \begin{align*}
    h'(\theta)
    &=\log2-\frac{2^\theta\log2}{2-2^\theta}-4+\frac{4}{1-\theta}\\
    &=\left(1-\frac{2^\theta}{2-2^\theta}\right)\log2+4\frac\theta{1-\theta}\\
    &=-\frac{2^\theta-1}{1-(2^\theta-1)}2\log2+4\frac\theta{1-\theta}\\
    &=2\sum_{j=1}^\infty\left[2\theta^j-\log(2)\left(2^\theta-1\right)^j\right]\\
    &\ge 2(2-\log2)\sum_{j=1}^\infty\theta^j,
    \end{align*}
    where the last step follows from the fact that $2^\theta-1\le\theta$ for $\theta\in[0,1]$.
    We conclude that $h'(0)=0$ and $h'(\theta)>0$ for $\theta\in(0,1)$, so that $h(\theta)>0$ for $\theta\in(0,1)$.
\end{proof}

The following lemma is a variant of the Bahadur--Ranga Rao bound for binomial tails, tailored to our needs:
\begin{lemma}\label{lem:bin-lower}
    For $1\le k< n$ and $p\in(0,1)$, we have
    \[
    \P(\text{Bin}(n,p)= k)\ge \frac{e^{-1/6}}{\sqrt{2\pi k}}\exp\left(-nD(k/n\|p)\right),
    \]
    where $D(x\|y)=x\log\frac{x}{y}+(1-x)\log\frac{1-x}{1-y}$ denotes the Kullback--Leibler divergence from $y$ to $x$.
\end{lemma}
\begin{proof}
    We have
    \[
        \P(\text{Bin}(n,p)= k)=\binom{n}{k}p^k(1-p)^{n-k}.
    \]
    We use the Stirling bounds
    \[
    \sqrt{2\pi m}\left(\frac me\right)^m\le m!\le \sqrt{2\pi m}\left(\frac me\right)^me^{\frac1{12m}},
    \]
    which hold for all $m\ge1$. These lead to
    \[
    \binom{n}{k}\ge \frac{e^{-\frac1{12}\left(\frac1k+\frac1{n-k}\right)}}{\sqrt{2\pi k(n-k)/n}}\frac{n^n}{k^k(n-k)^{n-k}}.
    \]
    We bound $\frac1k+\frac1{n-k}\le2$ and $(1-k/n)^{-1/2}\ge1$ to obtain
    \[
    \binom{n}{k}\ge \frac{e^{-1/6}}{\sqrt{2\pi k}}\exp\left(n\left(-\frac kn\log\frac kn-\left(1-\frac kn\right)\log\left(1-\frac kn\right)\right)\right).
    \]
    This leads to
    \[
    \binom{n}{k}p^k(1-p)^{n-k}\ge\frac{e^{-1/6}}{\sqrt{2\pi k}}\exp\left(-nD(k/n\|p)\right).
    \]
\end{proof}